\documentclass[aap]{imsart}

\RequirePackage{amsthm,amsmath,amsfonts,amssymb}
\RequirePackage[numbers]{natbib}
\RequirePackage[colorlinks,citecolor=blue,urlcolor=blue]{hyperref}
\RequirePackage{graphicx}

\startlocaldefs

\newtheorem{theorem}{Theorem}[section]

\newtheorem{remark}{Remark}[section]
\newtheorem{corollary}{Corollary}[section]
\newtheorem{proposition}{Proposition}[section]
\theoremstyle{remark}

\usepackage{tabularx}
\usepackage{longtable}
\usepackage{ulem}
\allowdisplaybreaks[4]

\newcommand{\bA}{{\bf A}}
\newcommand{\bB}{{\bf B}}

\newcommand{\um}{\underline{m}}

\newcommand{\bo}{{\bf 0}}

\newcommand{\bG}{{\bf G}}
\newcommand{\bH}{{\bf H}}

\newcommand{\bI}{{\bf I}}

\newcommand{\bT}{{\bf T}}

\newcommand{\bX}{{\bf X}}
\newcommand{\bU}{{\bf U}}
\newcommand{\bV}{{\bf V}}

\newcommand{\bz}{{\bf z}}

\newcommand{\bvare}{{\boldsymbol\varepsilon}}

\newcommand{\bqa}{\begin{eqnarray}}
\newcommand{\eqa}{\end{eqnarray}}
\newcommand{\bqn}{\begin{eqnarray*}}
\newcommand{\eqn}{\end{eqnarray*}}
\newcommand{\be}{\begin{equation}}
\newcommand{\ee}{\end{equation}}
\newcommand{\non}{\nonumber\\}

\newcommand{\rE}{{\rm E}}

\newcommand{\md}{\mbox{d}}

\newcommand{\mS}{{\bf S}}

\def\cov{{\rm Cov}}

\endlocaldefs

\begin{document}

\begin{frontmatter}
\title{Invariance principle and CLT for the spiked eigenvalues of large-dimensional
Fisher matrices 
and applications}
\runtitle{Invariance principle and CLT for spiked eigenvalues of Fisher matrices}

\begin{aug}
\author[A]{\fnms{Dandan} \snm{Jiang}\ead[label=e1]{jiangdd@xjtu.edu.cn}},
\author[B]{\fnms{Zhiqiang} \snm{Hou}\ead[label=e2]{houzq399@nenu.edu.cn}},
\author[C]{\fnms{Zhidong} \snm{Bai}\ead[label=e3]{baizd@nenu.edu.cn}}
\and
\author[D]{\fnms{Runze} \snm{Li}\thanksref{t1}\ead[label=e4]{rzli@psu.edu}}

\thankstext{t1}{Corresponding author: Runze Li.}
\runauthor{Jiang, Hou, Bai and Li}
\address[A]{
School of Mathematics and Statistics,
Xi'an Jiaotong University,
Xi'an {\rm 710049}, China.
\printead{e1}}

\address[B]{Shandong University of Finance and Economics, Jinan {\rm 250014}, China.
\printead{e2}}

\address[C]{KLASMOE and School of Mathematics and Statistics,
Northeast Normal University,
Changchun {\rm 130024}, China.
\printead{e3}}

\address[D]{
Department of Statistics, Pennsylvania State University, University Park,
PA 16802, USA.
\printead{e4}}
\end{aug}

\begin{abstract}
This paper aims to derive asymptotical distributions of {\color{blue} the spiked eigenvalues of the large-dimensional spiked}
 Fisher matrices without Gaussian assumption and the restrictive 
 assumptions on covariance matrices.
We first establish invariance principle for the spiked eigenvalues of the Fisher matrix. That is, we show
that the limiting distributions of the spiked eigenvalues are invariant
over a large class of population distributions satisfying certain conditions.
Using the invariance principle, we further established a central limit theorem (CLT) for
the spiked eigenvalues. As some interesting applications, we use the CLT to
derive the power functions of Roy Maximum root test for linear hypothesis in linear models  {\color {blue}and  the test in signal detection.}
We conduct some Monte Carlo simulation to compare the proposed test with existing ones.
\end{abstract}

\begin{keyword}[class=MSC2020]
\kwd[Primary ]{ 60B20}
\kwd{62H25}
\kwd[; secondary ]{60F05}
\end{keyword}

\begin{keyword}
\kwd{Roy maximum root test}
\kwd{Random matrix theory}
\kwd{Spiked model}
\kwd{Two-sample covariance problem}
\end{keyword}

\end{frontmatter}

\section{Introduction}\label{Int}

Motivated by several applications of hypothesis on two-sample covariance matrices
and linear hypothesis on regression coefficient matrix in linear models,
we consider the following spiked model.
Let  ${\boldsymbol\Sigma}_1$ and ${\boldsymbol\Sigma}_2$ be the covariance matrices from two $p$-dimensional populations, and
${\mathbf S}_1 ,{\mathbf S}_2$ be the corresponding
sample covariance matrices with sample sizes $n_1$ and $n_2$. The
two-sample spiked model assumes that  ${\boldsymbol\Sigma}_2={\boldsymbol\Sigma}_1+{\boldsymbol\Delta}$,
where ${\boldsymbol\Delta}$ is a $p \times p$ matrix of  finite rank $M$. It is
of great interest to study statistical inference on the spikes,
including but not limited to, testing  the presence of the spikes,
testing the number of the  spikes, and calculating the power under
the alternative hypothesis in two-sample testing problems. Thus,
it is critical to establish the asymptotic properties of the
spiked eigenvalues of a Fisher matrix  ${\mathbf F}={\mathbf S}_1{\mathbf S}_2^{-1}$. It is of particular
interest to derive the asymptotical distribution of
$\lambda_{\max}(\mathbf F)$, the largest eigenvalue of $\mathbf F$. In this
paper, we first establish an invariance principle for the spiked eigenvalues of
the Fisher matrix, and the invariance principle can be used as a
universal probability tool to derive the asymptotical distribution
of local spectral statistics of $\mathbf F$.

The one-sample spiked model ${\boldsymbol\Sigma}={\mathbf I}_p+{\boldsymbol \Delta}$ has received a lot
of attentions in the literatures, where $\mathbf I_p$ is the identity
matrix, and $\mathbf \Delta$ is a low-rank $p
\times p$ matrix. Since the seminal work \cite{Johnstone2001}, a
pioneer work on one-sample spiked model under the setting of the
dimension $p$ is of the same order of the sample size $n$, many works have
been published and  focus  on the research of  the asymptotic law
for the spiked eigenvalues of large-dimensional covariance matrix
\citep{BaiNg2002, Baik2005, BaikSilverstein2006, Paul2007,
BaiYao2008, Onatski2009,Onatski2012}. Also see  \citep{BaiYao2012,
FanWang2015, CaiHanPan2017,JiangBai2018} for a more general
one-sample spiked model. These works establish  the limiting distribution for the spiked eigenvalues
of the sample covariance matrices under different settings.

Compared with the one-sample spiked model, there are relatively
few studies on  two-sample spiked model. \cite{Zhengetal2017}
derived the limiting distribution of the eigenvalues of a
Fisher matrix and establish a CLT for a wide class of functionals of all the eigenvalues as a whole.
{\color{blue} According to \cite{Hanetal2016}, the largest eigenvalue of the Fisher matrix follows the
    Tracy-Widom law under some conditions}.
Therefore, the results in \cite{Zhengetal2017}
are not applicable for the local spectral statistics, especially
for those made of spiked eigenvalues. For a simplified two-sample spiked
models assuming that  ${\boldsymbol\Sigma}_1{\boldsymbol\Sigma}_2^{-1}$ is a rank $M$
perturbation of identity matrix with diagonal independence and
bounded population fourth moment,  \cite{WangYao2017} established CLT for the extreme
eigenvalues  of large-dimensional spiked Fisher matrices.
\cite{JO2019} proposed the rank-one two-sample spiked models that
represent the James' classes in \cite{James1964}, and further derived the asymptotic
behavior of the  likelihood ratios in large-dimensional setting.
The aforementioned works impose some restrictive or unrealistic conditions
such as only one threshold,  diagonal assumption, rank-one perturbation, etc.
The first two conditions (i.e., only one threshold and diagonal assumption) imply
that the spiked eigenvalues and non-spiked eigenvalues
are generated by independent variables.  The assumption of rank-one perturbation
means that there is only one spike.

In this paper, we study the asymptotical properties of Fisher
matrix of a two-sample spiked model under a new setting,
which allows that  (a) two populations may have very different
covariance matrices, (b) the spiked eigenvalues of the Fisher matrix
may be scattered into the spaces of a few bulks, and (c) the largest
eigenvalue may tend to infinity. Under this new setting, the fourth moments of population
are not required to be bounded. For ease of presentation, we refer to a Fisher matrix with
this new setting as a {\it generalized spiked Fisher matrix}.
Under a general  setting, we establish an invariance principle
for the generalized spiked Fisher matrix by using a similar but more complicated
technique in  \cite{JiangBai2018}.
With the invariance principle of generalized spiked Fisher matrix,
we establish a CLT for the local spiked eigenvalues of a generalized spiked Fisher
matrix. As {\color{blue}applications}, we use the CLT to
derive the power {\color{blue}functions} of Roy maximum root test on linear
hypothesis in large-dimensional linear models {\color{blue} and the signal detection test.}

Compared with  existing works on spiked Fisher matrices, this work
relaxes  the bounded fourth-moment condition on population to a
tail probability condition, which is a regular and necessary
condition in the weak convergence of the largest eigenvalue. Under
the new setting, we establish an invariance principle theorem for the
generalized spiked Fisher matrix. With the aid of the invariance principle theorem,
we further establish the CLT for the local spiked
eigenvalues of large-dimensional generalized spiked Fisher
matrices under mild assumptions on the population distribution.
As a by-product, our results naturally extend the result of \cite{WangYao2017} to
a general case under which we can successfully remove the diagonal or  block-wise
diagonal assumption on the matrix  ${\boldsymbol\Sigma}_1{\boldsymbol\Sigma}_2^{-1}$.  Our  setting allows
that the spiked eigenvalues may be generated from the variables
partially dependent on the ones corresponding to the non-spiked
eigenvalues. Our setting also allows a few pairs of thresholds for bulks of spiked
eigenvalues.  In summary, our setting is more realistic than the ones  in the existing works.

The rest of this paper is organized as follows. We establish the invariance principle and the CLT of generalized
spiked Fisher matrix in Sections 2 and 3, respectively.
We use the CLT to derive the local power functions of Roy maximum root test for linear hypothesis
in large-dimensional linear models {\color{blue} and the test in  signal detection}  in Section 4. We present some numerical study in Section 5.
Some technical proofs are given in Section 6. Additional technical proofs
are presented in the supplementary material.

\section{Invariance principle for spiked generalized Fisher matrix}
\label{Pre}
\subsection{Phase transition for the spiked eigenvalues}\label{Sec2.1}

Suppose that $\check{\mathbf x}_j$, $j=1, \ldots, n_1$ and
$\check{\mathbf y}_l$, $l=1,\ldots, n_2$ are random samples from
$p$-dimensional populations $\check{\mathbf x}$ and $\check{\mathbf y}$ with
$\rE(\check{\mathbf x})=\rE(\check{\mathbf y})=\mathbf 0$ and $\cov(\check{\mathbf x})=\boldsymbol\Sigma_1$
and $\cov(\check{\mathbf y})=\boldsymbol\Sigma_2$, respectively. Define
$\mathbf x_j=\boldsymbol \Sigma_1^{-1/2} \check{\mathbf x}_j$ and $ \mathbf y_l=\boldsymbol\Sigma_2^{-1/2}
\check{\mathbf y}_l$, for $j=1,\cdots, n_1$ and $l=1,\cdots,n_2$. For ease of presentation, we represent the
samples as $\boldsymbol\Sigma_1^{1/2}\mathbf X$ and $\boldsymbol\Sigma_2^{1/2}\mathbf Y$,  where
$\mathbf X=(\mathbf x_1,\cdots,\mathbf x_{n_1})=(x_{ij})$ and
$\mathbf Y=(\mathbf y_1,\cdots,\mathbf y_{n_2})=(y_{il})$ are two independent
$p$-dimensional arrays with  components having zero mean  and
identity covariance matrix. To broaden the application of new theory
developed in this paper, we allow both $\mathbf X$ and $\mathbf Y$ to be complex matrix, and
denote $\mathbf A^*$ to be the conjugate 
transpose of a complex matrix $\mathbf A$.
Define $\mathbf T_p=\boldsymbol\Sigma_1^{1/2}\boldsymbol\Sigma_2^{-{1/2}}$ and assume that the
spectrum of  $\mathbf T_p^*\mathbf T_p$  is formed as
\begin{equation}
\beta_{p,1}, \cdots,  \beta_{p,j},\cdots,\beta_{p,p}\label{array}
\end{equation}
in descending order.
Denote the spikes as
$\alpha_k:=\beta_{p,j_k+1}= \cdots= \beta_{p, j_k+m_k}$ with $j_k$'s being arbitrary ranks in the array (\ref{array}). 
All the spikes,
 $\alpha_1, \cdots, \alpha_K$, with multiplicity $m_k, k=1,\cdots,K$  are lined arbitrarily in groups among all the eigenvalues, satisfying $m_1+\cdots+m_K=M$, a fixed integer.
{\color{blue}Additionally,  the spiked eigenvalues are allowed  to be
extremely larger or smaller than the non-spiked ones, which is
useful when considering system fluctuations. Such settings have
not been considered in other existing literatures.}
The corresponding sample covariance matrices of the two observations are
\begin{equation}{\mathbf S}_1
=\frac{1}{n_1}\boldsymbol\Sigma_1^{1 \over 2}
\mathbf X\mathbf X^*\boldsymbol\Sigma_1^{1 \over 2}, \quad {\mathbf S}_2
=\frac{1}{n_2}\boldsymbol\Sigma_2^{1 \over 2}
\mathbf Y\mathbf Y^*\boldsymbol\Sigma_2^{1 \over 2}, \label{S1}
\end{equation}
respectively.
We will investigate  the eigenvalues of the generalized spiked Fisher matrix
$\mathbf F=\mathbf S_1\mathbf S_2^{-1}=\boldsymbol\Sigma_1^{1/2}\tilde {\mathbf S}_1
\boldsymbol\Sigma_1^{1/2}\boldsymbol\Sigma_2^{-{1/2}}\tilde {\mathbf S}_2^{-1}\boldsymbol\Sigma_2^{-{1/2}}$,
where
${\tilde {\mathbf S}}_1=(1/n_1)\mathbf X\mathbf X^*$ and ${\tilde {\mathbf S}}_2=({1}/{n_2})\mathbf Y\mathbf Y^*$  are the standardized
sample covariance matrices, respectively.
It is well known that the eigenvalues of $\mathbf F$ are the same of the matrix with the form (Still use $\mathbf F$ for brevity, if no confusion):
\begin{equation}
\mathbf F=\mathbf T_p^{*}{\tilde {\mathbf S}}_1\mathbf T_p{\tilde {\mathbf S}}_2^{-1}.\label{F}
\end{equation}
Define the singular value decomposition of ${\mathbf T}_p$ as
 \be \mathbf T_p= \mathbf U\left(
\begin{array}{cc}
 {\mathbf D}_1^{1 \over 2} & \bo    \\
 \bo & {\mathbf D}_2^{1 \over 2}
\end{array}
\right){\mathbf V}^{*},
\label{SVD}
\ee
where $\mathbf U, \mathbf V$ are  unitary matrices (orthogonal matrices for real case),
${\mathbf D}_1$ is a diagonal matrix of the $M$ spiked eigenvalues of  the generalized spiked Fisher matrix $\mathbf F$,
 and ${\mathbf D}_2$ is the diagonal matrix of the non-spiked ones  with bounded components.

 Let $\mathcal{J}_k=\{ j_k+1,\cdots, j_k+m_k\}$ be the set of  ranks  of  $\alpha_k$ with multiplicity $m_k$ among all the eigenvalues of $\mathbf T_p^*\mathbf T_p$.
Set the sample eigenvalues of the generalized spiked  Fisher matrix $\mathbf F$ in the descending order  as
$l_{p,1}(\mathbf F), \cdots,  l_{p,j}(\mathbf F), \cdots,  l_{p,p}(\mathbf F)$.
To derive the limiting law for the spiked eigenvalues of  $\mathbf F$, we impose some
conditions as follow.
\begin{description}
\item{\bf Assumption~1} The two double arrays $\{x_{ij}, i, j =
1,2,...\} $ and $\{y_{ij}, i, j = 1,2,...\}$ consist of independent
and identically distributed ({\rm i.i.d.}) random variables with
mean 0 and variance 1.
{Furthermore, $\rE x_{ij}^2=0$ and $\rE y_{ij}^2=0$  hold for the complex case
if the variables and $\bT_p$ are complex.}
\item{\bf Assumption~2.}
Assuming that
$c_{n_1}\!=\!p/n_1 \in (0, \infty)$, $c_{n_2}\!=\!p/n_2 \in (0,1)$ is considered throughout the paper when  $\min( p,n_1,n_2)\rightarrow \infty$.
The matrix $\mathbf T_p=\boldsymbol\Sigma_1^{1/2}\boldsymbol\Sigma_2^{-{1/2}}$ is non-random and the empirical spectral distribution of $\{\mathbf T_p^*\mathbf T_p\}$ excluding  the spikes,
$H_n(t)$, tends to proper probability measure $H(t)$  if  $\min( p,n_1,n_2)\rightarrow \infty$.
(\textbf{Dandan: why $c_2$ has to be in (0,1)})
\item{\bf Assumption~3.}
Assume $\lim\limits_{\tau \rightarrow \infty}\tau^4 {\rm P}\left(|x_{11}| >\tau \right)=0$  and
$\lim\limits_{\tau \rightarrow \infty}\tau^4 {\rm P}\left(|y_{11}| >\tau \right)$ $= 0$
for the {\rm i.i.d.} samples, where both of the fourth-moments are not necessarily required to exist.
\item{\bf Assumption~4}
Suppose that
 \begin{eqnarray*}
 &&\max\limits_{t,s} |u_{ts}|^2\big[\rE\{|x_{11}|^4{\delta}(|x_{11}|<\eta_{n_1}\sqrt{n_1})\}-{\mu_1}\big] \rightarrow 0, \\ 
 &&\max\limits_{t,s} |v_{ts}|^2\big[\rE\{|y_{11}|^4\delta(|y_{11}|<\eta_{n_2}\sqrt{n_2})\}-{\mu_2}\big]\rightarrow 0, 
 \end{eqnarray*}
 where for some constants $\mu_1$ and $\mu_2$,
$\delta ( \cdot )$ is
the indicator function and $\mathbf U_1=(u_{ts})$, $\mathbf V_1=(v_{ts})$
are    the first $M$ columns  of matrix $\mathbf U$ and $\mathbf V$ defined in (\ref{SVD}), respectively.
\item{\bf Assumption~5}
The spiked eigenvalues of the matrix  $\mathbf F$, $\alpha_1,\cdots, \alpha_K$,
with multiplicities $m_1,\cdots,$  $m_K$  laying outside the support of $H(t)$, satisfy  $\psi'(\alpha_k)>0, $  for $1\le k \le K$, where \\[-4mm]
\be
\psi_k:= \psi(\alpha_k)=\frac{\alpha_k\left\{1-c_{1}\int t/(t-\alpha_k)\mbox{d}H(t)\right\}}
{1+c_{2}\int\alpha_k/(t-\alpha_k)\mbox{d}H(t)}\label{psik}
\ee
is the limit of the distant sample spiked eigenvalues of a generalized spiked Fisher matrix.
\end{description}

The limit in (\ref{psik}) is derived by \cite{Jiangetal2019} when $x_{ij}$ and $y_{ij}$
have bounded fourth moments, and is detailed in the following proposition.

\begin{proposition}\label{PT}
For any spiked eigenvalue $\alpha_k$ with multiplicity $m_k, k=1,\cdots,K$ of $\mathbf F$ defined in (\ref{F}),
 let
  \[\rho_k =
\left\{
\begin{array}{cl}
  \psi(\alpha_k),& \mbox{ if } \psi'(\alpha_k)>0;  \\
   \psi(\underline\alpha_k), &   \mbox{ if there exists } \underline\alpha_k \mbox{ such that } \psi'(\underline\alpha_k)=0,
   \mbox{ and }\psi'(t)<0,  \mbox{ for all } \alpha_k\le t<\underline{\alpha}_k;
\\
   \psi(\overline\alpha_k), &   \mbox{ if there exists } \overline\alpha_k \mbox{ such that } \psi'(\underline\alpha_k)=0,
    \mbox{ and }
 \psi'(s)<0, \mbox{ for all } \overline{\alpha}_k<s\le\alpha_k,
\end{array}
\right.
\]
where $\psi(\alpha_k)$ is defined in (\ref{psik}).
Then under the Assumptions~1-2, and the bounded fourth-moment assumption, it holds that
for all $ j \in \mathcal{J}_k$, $\{l_{p,j}\}$ almost surely that  $\{l_{p,j}/\rho_k-1\}$  converges to 0.
 \end{proposition}

Since the convergence of $c_{n_1} \to c_1$, $c_{n_2} \to c_2$  and $H_n(t) \to H(t)$ may be very slow, the difference $\sqrt{n}(l_{p,j} -\psi_k)$ may not have a limiting distribution.
Thus, we use \\[-5mm]
\[
  \psi_{n,k}:= \psi_n(\alpha_k)=\frac{\alpha_k\left\{1-c_{n_1}\int
  t/(t-\alpha_k)\mbox{d}H_n(t)\right\}}{1+c_{n_2}\int\alpha_k/(t-\alpha_k)\mbox{d}H_n(t)},\]
   instead of $\psi_k$ in $\rho_k$,
and $n$ denotes $(n_1,n_2)$, especially the case of CLT. Then,
we only require $c_{n_1}=p/n_1; c_{n_2}=p/n_2$, and both the dimensionality $p$ and the sample sizes $n_1, n_2$ grow to infinity  simultaneously, but not necessarily  in proportion.

Note that the Proposition~\ref{PT}  holds under the bounded fourth-moment assumption.
To relax the bounded fourth-moment assumption to  Assumption~3 on the tail probability,
we introduce the following  truncation procedures.

Let $\delta(A)$ be the indication of set $A$. Let $\hat x_{ij}= x_{ij} \delta (|x_{ij}| < \eta_{n_1} \sqrt{n_1})$,
$\tilde x_{ij}=(\hat x_{ij} - \rE \hat x_{ij})/ \sigma_{n_1}$ and $\hat y_{ij}= y_{ij}
\delta (|y_{ij}| < \eta_{n_2} \sqrt{n_2})$,  $\tilde y_{ij}=(\hat y_{ij} -
\rE \hat y_{ij})/ \sigma_{n_2}
$,
where  $\sigma_{n_1}^2=$ $\rE \left|\hat x_{ij}- \rE  \hat x_{ij} \right|^2$ and $\sigma_{n_2}^2=$ $\rE \left|\hat y_{ij}- \rE  \hat y_{ij} \right|^2$.
By the related techniques of the proofs in Supplement~B of \cite{JiangBai2018},
we can show that it is equivalent to  replace the entries of $x_{ij}, y_{ij}$
with their corresponding  truncated and centralized variables by  Assumption~3.
In addition,
the convergence rates of arbitrary moments of  $\tilde x_{ij}$ and $\tilde y_{ij}$
are  the same as the ones depicted in Lemma~C.1 in \cite{JiangBai2018}.
Therefore, it is reasonable to consider the generalized spiked Fisher matrix $\mathbf F=\mathbf S_1 \mathbf S_2^{-1}$,
which is  generated from the entries truncated at $\eta_{n_1} \sqrt{n_1}$ for $x_{ij}$ and
$\eta_{n_2} \sqrt{n_2}$ for $y_{ij}$, centralized and renormalized. For simplicity, we assume that
$|x_{ij}| <\eta_{n_1} \sqrt{n_1}, |y_{ij}| <\eta_{n_2} \sqrt{n_2}$,
$\rE x_{ij}=\rE y_{ij}=0, \rE |x_{ij}|^2=\rE |y_{ij}|^2=1$ for the real case and
Assumption~3
is satisfied. But for the complex case, the truncation and renormalization cannot reserve
{the requirement of $\rE x_{ij}^2=\rE y_{ij}^2 =0$. However, one may prove
that  $\rE x_{ij}^2 =o(n_1^{-1})$ and $\rE y_{ij}^2 =o(n_2^{-1})$.}
By the truncation procedures, the Proposition~\ref{PT} still holds
in probability without the bounded fourth-moment assumption if the Assumption~3  is satisfied.

\subsection{Invariance principle theorem}
 For the generalized  spiked Fisher matrix
$\mathbf F=\mathbf T_p^{*}{\tilde {\mathbf S}}_1\mathbf T_p{\tilde {\mathbf S}}_2^{-1}$ defined in (\ref{F}),
we consider the arbitrary sample spiked eigenvalue of $\mathbf F$, $l_{p,j}, j \in \mathcal{J}_k, k=1,\ldots, K$.
By the singular value decomposition of $\mathbf T_p$ in (\ref{SVD}) and  the eigen-equation for $\mathbf F$, we have
\[0=|l_{p,j}\mathbf I- \mathbf F|=\left|l_{p,j}\mathbf I-\mathbf V\left(
\begin{array}{cc}
 \mathbf D^{1 \over 2}_1 & \bo    \\
 \bo & \mathbf D^{1 \over 2}_2
\end{array}
\right)
\mathbf U^* {\tilde {\mathbf S}}_1\mathbf U
\left(
\begin{array}{cc}
 \mathbf D^{1 \over 2}_1 & \bo    \\
 \bo & \mathbf D^{1 \over 2}_2
\end{array}
\right)\mathbf V^*{\tilde {\mathbf S}}^{-1}\right|.\]
It is equivalent to
\begin{eqnarray*}
0&=&|l_{p,j} {\mathbf V}^*{\mathbf V}- {\rm diag}({\mathbf D}_1^{1 \over 2},{\mathbf D}_2^{1 \over 2}){\mathbf U}^* {\tilde {\mathbf S}}_1{\mathbf U}{\rm diag}({\mathbf D}_1^{1 \over 2},{\mathbf D}_2^{1 \over 2})|\\
&=&\Bigg|\begin{pmatrix} l_{p,j} {\mathbf V}_1^*{\tilde {\mathbf S}}_2 {\mathbf U}_1, &l_{p,j} {\mathbf V}_1^* {\tilde {\mathbf S}}_2 {\mathbf V}_2\\
l_{p,j} \bV_2^*{\tilde {\mathbf S}}_2 {\mathbf V}_1, &l_{p,j} {\mathbf V}_2^* {\tilde {\mathbf S}}_2 \bV_2\end{pmatrix}
-\begin{pmatrix}{\mathbf D}_1^{1 \over 2}{\mathbf U}_1^*{\tilde {\mathbf S}}_1{\mathbf U}_1\mathbf D_1^{1 \over 2}, & {\mathbf D}_1^{1 \over 2}{\mathbf U}_1^*
{\tilde {\mathbf S}}_1\mathbf D_2^{1 \over 2}\\ \mathbf D_2^{1 \over 2}\mathbf U_2^* {\tilde {\mathbf S}}_1 \mathbf U_1 \mathbf D_1^{1 \over 2},
& \mathbf D_2^{1 \over 2}\mathbf U_2^* {\tilde {\mathbf S}}_1\mathbf D_2^{1 \over 2}\end{pmatrix}\Bigg|.
\end{eqnarray*}

If we consider $l_{p,j}$ as  an sample spiked eigenvalue of $\mathbf F$,  but not of $\mathbf D_2^{1/2}\mathbf U_2^* {\tilde {\mathbf S}}_1\mathbf U_2\mathbf D_2^{1/2}$ $\cdot(\mathbf V_2^* \tilde {\mS_2} \mathbf V_2\!)\!^{-1}$, then the following equation holds for every sample spiked eigenvalue, $l_{p,j}, j \in \mathcal{J}_k, k=1,\cdots, K$
\begin{align}
0=&\Big|l_{p,j} \mathbf V_1^*{\tilde {\mathbf S}}_2\mathbf  V_1-\mathbf D_1^{1 \over 2}\mathbf U_1^* {\tilde {\mathbf S}}_1\mathbf U_1\mathbf D_1^{1 \over 2}-(l_{p,j} \mathbf V_1^* {\tilde {\mathbf S}}_2\mathbf V_2-\mathbf D_1^{1 \over 2}\mathbf U_1^* {\tilde {\mathbf S}}_1\mathbf U_2\mathbf D_2^{1 \over 2})\mathbf Q^{-{1 \over 2}}\\
&\! \cdot \!(l_{p,j} \mathbf I\!-\!\mathbf Q^{-{1 \over 2}}\mathbf D_2^{1 \over 2}\!\mathbf U_2^*{\tilde {\mathbf S}}_1\mathbf U_2\mathbf D_2^{1 \over 2}\mathbf Q^{-{1 \over 2}})^{-1}\mathbf Q^{-{1 \over 2}}(l_{p,j} \mathbf V_2^*{\tilde {\mathbf S}}_2\mathbf  V_1\!-\!\mathbf D_2^{1 \over 2}\mathbf U_2^* {\tilde {\mathbf S}}_1 \mathbf U_1 \mathbf D_1^{1 \over 2})\Big|,\non
=&\bigg|
\frac{l_{p,j}}{n_2}\mathbf V_1^*\mathbf Y \mathbf Y^*\mathbf V_1-\frac{l_{p,j}^2}{n_2^2}\mathbf V_1^*\mathbf Y\mathbf Y^*\mathbf V_2\mathbf Q^{-{1 \over 2}} (l_{p,j} \mathbf I- \tilde {\mathbf F}\big)^{-1} \mathbf Q^{-{1 \over 2}}\mathbf V_2^*\mathbf Y \mathbf Y^*\mathbf V_1\non
&-\frac{l_{p,j}}{n_1}\mathbf  D_1^{1 \over 2}\mathbf U_1^*\mathbf X\big(l_{p,j}\mathbf I-\underline{\tilde {\mathbf F}}\big)^{-1}\mathbf X^*\mathbf U_1\mathbf D_1^{1 \over 2}\non
&+\frac{l_{p,j}}{n_2}\mathbf V_1^*\mathbf Y \mathbf Y^*\mathbf V_2 \mathbf Q^{-{1 \over 2}}\big(l_{p,j} \mathbf I- \tilde {\mathbf F}\big)^{-1} \mathbf Q^{-{1 \over 2}}
\frac{1}{n_1}\mathbf D_2^{1 \over 2}\mathbf U_2^*\mathbf X\mathbf X^* \mathbf U_1 \mathbf D_1^{1 \over 2}\non
&+\frac{l_{p,j}}{n_1}\mathbf D_1^{1 \over 2}\mathbf U_1^* \mathbf X\mathbf X^*\mathbf U_2\mathbf D_2^{1 \over 2}\mathbf Q^{-{1 \over 2}} (l_{p,j} \mathbf I- \tilde {\mathbf F})^{-1}\mathbf Q^{-{1 \over 2}}\frac{1}{n_2} \mathbf V_2^*\mathbf Y\mathbf Y^* \mathbf V_1\bigg|,\non
=&\bigg|\{\psi_{n,k}+c_{2}\psi^2_{n,k}m(\psi_{n,k})\}\mathbf I_M+\psi_{n,k}\um(\psi_{n,k})\mathbf D_1+\frac 1{\sqrt{p}}\gamma_{kj}\psi_{n,k}\mathbf I_M\non
&\quad+\mathbf B_1(l_{p,j})+\mathbf B_2(l_{p,j})+\frac{\psi_{n,k}}{\sqrt{p}}\boldsymbol \Omega_{M}(\psi_{n,k},\mathbf X,\mathbf Y)+o(\frac {\psi_{n,k}}{\sqrt{p}})\bigg|,
\label{IBJ}
\end{align}
where $\mathbf Q=\mathbf V_2^* {\tilde {\mathbf S}}_2\mathbf V_2$, and $ \tilde {\mathbf F}$ and $\underline{\tilde {\mathbf F}}$ are defined as
\be
 \tilde {\mathbf F}=\displaystyle\frac 1 {n_1}\mathbf Q^{-{1 \over 2}}\mathbf D_2^{1 \over 2}\bU_2^* \mathbf X\mathbf X^*\mathbf U_2\mathbf D_2^{1 \over 2}\mathbf Q^{-{1 \over 2}}, \quad \underline{{\tilde {\mathbf F}}}=\displaystyle\frac 1 {n_1}\mathbf X^*\mathbf U_2\mathbf D_2^{1 \over 2} \mathbf Q^{-1}\mathbf D_2^{1 \over 2}\mathbf U_2^* \mathbf X;\label{FFdef}
\ee
 $\psi_{n,k}$  is used instead of $\psi_k$  to  avoid  the slow convergence 
as mentioned  in  Section~2.1, $\gamma_{kj}=\sqrt{p} \big({l_{p,j}}/\psi_{n,k}-1\big),~  j \in \mathcal{J}_k$ and 
\begin{align*}
\mathbf  B_{1}(l_{p,j})
&= \frac{\psi^2_{n,k}}{n_2^{2}}
\mathbf V_1^*\mathbf Y\mathbf Y^*\mathbf V_2\mathbf Q^{-\frac1{2}} (\psi_{n,k} \mathbf I-\tilde {\mathbf F})^{-1} \mathbf Q^{-\frac1{2}}\mathbf V_2^*\mathbf Y\mathbf Y^*\mathbf V_1\non
&\quad-\frac{l^2_{p,j}}{n_2^{2}}
\mathbf  V_1^*\mathbf Y \mathbf Y^*\mathbf V_2\mathbf Q^{-\frac1{2}} (l_{p,j} \mathbf I-\tilde {\mathbf F})^{-1}\mathbf Q^{-\frac1{2}}\mathbf V_2^*\mathbf Y\mathbf Y^*\mathbf V_1;\non
\mathbf B_{2}(l_{p,j})
&=\!\frac{\psi_{n,k}}{n_1}\!\mathbf D_1^{1 \over 2}\!\mathbf U_1^*\mathbf X\big(\psi_{n,k}\mathbf I\!-\! \underline{\tilde {\mathbf F}}\big)^{-1}\!\mathbf X^*\mathbf U_1\mathbf D_1^{1 \over 2}\!-\!\frac{l_{p,j}}{n_1}\!\mathbf D_1^{1 \over 2}\!\mathbf U_1^*\mathbf X\big(l_{p,j}\mathbf I\!-\! \underline{\tilde {\mathbf F}}\big)^{-1}{\mathbf X}^*{\mathbf U}_1\mathbf D_1^{1 \over 2}.\nonumber
\end{align*}
Moreover,  the $ \boldsymbol\Omega_{M}(\lambda,\mathbf X,\mathbf Y)$ is defined as
\begin{align}
 \boldsymbol\Omega_{M}(\lambda,\mathbf X,\mathbf Y)&=\sum\limits_{j=1}^5\boldsymbol\Omega_{M,j}(\lambda,\mathbf X,\mathbf Y),\label{OmegaM}
\end{align}
where
\begin{align}
 \boldsymbol\Omega_{M,1}(\lambda,\!\mathbf X,\!\mathbf Y)&\!=\!\sqrt{p}\mathbf V_1^*( {\tilde {\mathbf S}}_{2}-\mathbf I_p)\mathbf V_1\non
  \boldsymbol\Omega_{M,2}(\lambda,\!\mathbf X,\!\mathbf Y)&\!=\!\frac{\sqrt{p}\lambda}{n_2}\!\Big\{\!{\rm tr}\! (\lambda \mathbf I\!-\!\tilde {\mathbf F})^{-1}\mathbf I
 \!-\!\frac{1}{n_2}
 \!\mathbf V_1^*\mathbf Y\mathbf Y^*\mathbf V_2\mathbf Q^{-\frac1{2}} \!(\lambda \mathbf I\!-\!\tilde {\mathbf F})^{-1} \!\mathbf Q^{-\frac1{2}}\!\mathbf V_2^*\mathbf Y\mathbf Y^*\mathbf V_1\!\Big\}\non
\boldsymbol\Omega_{M,3}(\lambda,\!\mathbf X,\!\mathbf Y)&\!=\!\frac{\sqrt{p}}{\sqrt{n_1}\lambda}\mathbf D_1^{1 \over 2} \!\Big[\!\frac{\lambda }{\sqrt{n_1}}\big\{{\rm tr}\! (\lambda \mathbf I\!-\! \underline{\tilde {\mathbf F}})^{-1} \mathbf I  \!-\! \mathbf U_1^*\bX\big(\lambda\mathbf I\!-\! \underline{\tilde {\mathbf F}}\big)^{-1}\!\mathbf X^*\mathbf U_1\!\big\}\Big ]\!\mathbf D_1^{1 \over 2}\non
\boldsymbol\Omega_{M,4}(\lambda,\!\mathbf X,\!\mathbf Y)&\!=\!\frac{\sqrt{p}}{n_1n_2}\mathbf V_1^*\mathbf Y\mathbf Y^*\mathbf V_2 \mathbf Q^{-{1 \over 2}}\big(\lambda \mathbf I\!-\!\tilde {\mathbf F}\big)^{-1} \mathbf Q^{-{1 \over 2}}
\mathbf D_2^{1 \over 2}\mathbf U_2^*\mathbf X\mathbf X^*\mathbf U_1 \mathbf D_1^{1 \over 2}\non
 \boldsymbol\Omega_{M,5}(\lambda,\!\mathbf X,\!\mathbf Y)&\!=\! \frac{\sqrt{p}}{n_1n_2}\mathbf D_1^{1 \over 2}\mathbf U_1^* \mathbf X\mathbf X^*\mathbf U_2\mathbf D_2^{1 \over 2} \mathbf Q^{-{1 \over 2}}\big(\lambda \mathbf I-\tilde {\mathbf F}\big)^{-1}\mathbf  Q^{-{1 \over 2}}
\mathbf V_2^*\mathbf Y\mathbf Y^* \mathbf V_1.\nonumber
\end{align}

Note that
 the covariance matrix between $\mathbf U_1^*\mathbf X$ and  $\mathbf V_1^*\mathbf Y$ is a zero matrix $\bo_{M\times M}$, then
according to   Lemma~2.7 in \cite{BaiSilverstein1998} and equation (9)  in \cite{Jiangetal2019}), 
we also
 obtain that $\psi_k$ satisfies the following equation
\begin{equation}
\psi_{k}+c_{2}\psi^2_{k}m(\psi_{k})+\psi_k\um(\psi_{k}) \alpha_k=0,\label{0eqa} 
\end{equation}
where $m (\lambda), \um(\lambda)$ are the Stieltjes transforms of the limiting spectral distributions  of $\tilde {\mathbf F}$ and $\underline{\tilde {\mathbf F}}$ defined in (\ref{FFdef}), respectively.

We establish the invariance principle of the generalized spiked Fisher matrix
in the following theorem. The invariance principle implies that the limiting distribution of
the spiked eigenvalues of a generalized spiked Fisher matrix remain
the same provided that the population distributions satisfy the Assumptions~1--5.

\begin{theorem}({\sc Invariance Principle Theorem})
\label{thm2}
Assuming that $(\mathbf X, \mathbf Y)$ and $(\mathbf W,\mathbf Z)$ are two pairs of double arrays, each
of which  satisfies Assumptions~1--5, 
$(\mathbf X, \mathbf Y)$ and  $(\mathbf W,\mathbf Z)$ have the same
$\mu_1, \mu_2$ in Assumption~4,
then $\boldsymbol\Omega_{M}(\lambda,\mathbf X,\mathbf Y)$ and $\boldsymbol\Omega_{M}(\lambda,\mathbf W,\mathbf Z)$ have the same
limiting distribution, provided that one of them has a limiting distribution. \end{theorem}

The proof of Theorem~\ref{thm2} is  given in Section 6.  
According to  Theorem \ref{thm2}, we may assume that $\mathbf X$ and $\mathbf Y$ are consist of entries with {\rm i.i.d.} {\color{blue}standard normal} variables in deriving the limiting distributions of $\mathbf B_{1}(l_{p,j})$, $\mathbf B_{2}(l_{p,j})$ and  $\boldsymbol\Omega_M(\psi_{n,k},\mathbf X,\mathbf Y)$.
Firstly,
define $m_2(\lambda)=\int {1}/{(\lambda-x)^2} \md \tilde{F}(x)$,
    $\underline{m}_2(\lambda)=\int {1}/{(\lambda-x)^2} \md \underline{\tilde F}(x)$,
  where $\tilde{F}(x)$   and $\underline{\tilde F}(x)$  are  the limiting spectral distributions of  the matrices $\tilde {\mathbf F}$ and $\underline{ \tilde {\mathbf F}}$, respectively.
Then, by the formula $\mathbf A^{-1}- \mathbf B^{-1}=\mathbf A^{-1}(\mathbf B-\mathbf A)\mathbf B^{-1}$ for any two invertible square matrices $\mathbf A$ and $\mathbf B$, we obtain that
 \begin{align*}
\mathbf B_{1}(l_{p,j})
&=\!\frac 1{\sqrt{p}}\!\gamma_{kj}\Big\{ c_{2}\psi^3_{n,k}m_2(\psi_{n,k}) +2 c_2\psi^2_{n,k}m(\psi_{n,k})\Big\}\mathbf I_{M}+
o(\frac {\psi_{n,k}}{\sqrt{p}});\\
\mathbf B_{2}(l_{p,j})
&=\!\frac 1{\sqrt{p}}\!\gamma_{kj}\Big\{\psi^2_{n,k}\um_2(\psi_{n,k})+\psi_{n,k} \um(\psi_{n,k})
\Big\}\mathbf D_1+
o(\frac {\psi_{n,k}}{\sqrt{p}})
\end{align*}
 Thus, it follows from the  equation (\ref{IBJ})
 \begin{align}
0&=\bigg|\{\psi_{n,k}+c_{2}\psi^2_{n,k}m(\psi_{n,k})\}\bI_M+\psi_{n,k}\um(\psi_{n,k})\mathbf D_1\non
& +\frac 1{\sqrt{p}}\gamma_{kj}\bigg[
\big\{ \psi_{n,k}+c_{2}\psi^3_{n,k}m_2(\psi_{n,k}) +2 c_2\psi^2_{n,k}m(\psi_{n,k})\big\}\mathbf I_{M}\non
&+\big\{\psi^2_{n,k}\um_2(\psi_{n,k})+\psi_{n,k} \um(\psi_{n,k})\big\}\mathbf D_1\bigg]+\frac{\psi_{n,k}}{\sqrt{p}}\boldsymbol\Omega_M(\psi_{n,k},\mathbf X,\mathbf Y)+o(\frac {\psi_{n,k}}{\sqrt{p}}) \bigg|.
\label{eigeneq8}
\end{align}

Furthermore,  the limiting distribution of $\boldsymbol\Omega_M(\psi_{n,k},\mathbf X,\mathbf Y)$  is derived in the following
  corollary by replacing  the  entries in $\mathbf X$ and $\mathbf Y$ with the {\rm i.i.d.} {\color{blue}standard normal} variables.
 The detailed  proof is in Section~\ref{SOmega}.

\begin{corollary}\label{coro1}
Suppose that  both $\mathbf X$ and $\mathbf Y$ satisfy the Assumptions~1-5, 
and let
\begin{align}
&\theta_k\!=\!c_2\!+\! c_2^2\psi_{k}^2 m_2(\psi_{k})\!+\!2c_2^2\psi_km(\psi_{k})\!+\! c_1\alpha_k^2\um_2(\psi_{k})\!+\! 2c_1 c_2 \alpha_k m_3(\psi_{k}),\label{theta}
\end{align}
where  $m_3(\lambda)=\int x/{(\lambda-x)^2} \md \tilde {F}(x)$.
Then, it holds that
 $\boldsymbol\Omega_M(\psi_{n,k}, \mathbf X,\mathbf Y)$ tends to a limiting distribution of an $M\times M$ Hermitian matrix $\boldsymbol\Omega_{\psi_{k}}$,  where
 ${\theta_k}^{-1/2}\left[\boldsymbol\Omega_{\psi_{k}}\right]_{kk}$ is  Gaussian Orthogonal Ensemble (GOE) for the real case, with
the entries above the diagonal being ${\rm i.i.d.} \mathcal{N}(0,1)$ and the entries on the diagonal being ${\rm i.i.d.} \mathcal{N}(0,2)$.
For the complex case, the ${\theta_k}^{-1/2}\left[\boldsymbol\Omega_{\psi_{k}}\right]_{kk}$ is
Gaussian Unitary Ensemble (GUE),
whose diagonal entries  are  {\rm i.i.d.} real $ \mathcal{N}(0,1)$, and the off diagonal entries are {\rm i.i.d.} complex $\mathcal{CN}(0,1)$.
\end{corollary}

\begin{remark}
If the Assumption~4 
is weaken to the following  Assumption~4', 
\bqn
\beta_{x,i_1j_1i_2j_2}&\!\!=\!\!&\lim \sum_{t=1}^p\!\bar{u}_{ti_1}\!u_{tj_1}\!u_{ti_2}\!\bar u_{tj_2}\!\big[\rE\{|x_{11}|^4\!\delta(|x_{11}|\!\le\! \sqrt{n_1})\}\!-\!2\!-\!q\big] \!<\! \infty,\\
\beta_{y,i_1j_1i_2j_2}&\!\!=\!\!&\lim \sum_{t=1}^p\!\bar v_{ti_1}\!v_{tj_1}\!v_{ti_2}\!\bar v_{tj_2}\!\big[\rE\{|y_{11}|^4\delta(|y_{11}|\!\le\! \sqrt{n_2})\}\!-\!2\!-\!q\big]\!< \!\infty,
\eqn
where $q=1$ for real case and 0 for complex,
${\mathbf u}_i=(u_{1i}, \ldots, u_{pi})'$  is the $i$th column of the matrix $\mathbf U_1$ and
${\mathbf v}_j=(v_{1j}, \ldots, v_{pj})'$ is the $j$th column of the matrix $\mathbf V_1$,
then all the conclusions of Corollary~\ref{coro1} still holds, but the limiting distribution of  $\boldsymbol\Omega_M(\psi_{n,k}, \mathbf X,\mathbf Y)$ turns to
 an $M\times M$ Hermitian matrix $\boldsymbol\Omega_{\psi_{k}}=(\omega_{ij})$,  which has the independent Gaussian entries of
 mean 0 and variance
\[{\rm Cov}(\omega_{i_1,j_1},\omega_{i_2,j_2})=\left\{
\begin{array}{ll}
(q+1)\theta_k+{\beta_{x,iiii}}\nu_{1}+{\beta_{y,iiii}}\nu_{2},  &    i_1=j_1=i_2=j_2=i; \\
\theta_k+{\beta_{x,ijij}}\nu_1+{\beta_{y,ijij}}\nu_2,  &      i_1=i_2=i\neq j_1=j_2=j;\\
{\beta_{x,i_1j_1i_2j_2}}\nu_1+{\beta_{y,i_1j_1i_2j_2}}\nu_2,  &  \mbox{other cases}.
\end{array}
\right.\]
Here  $\theta_k$ is defined in (\ref{theta}),
$\nu_1=c_1\alpha_k^2\um^2(\psi_{k})$
and
$\nu_2=c_2\big\{1+c_2\psi_km(\psi_k)\big\}^2$.
 \label{rmkwij}
  \end{remark}

The proof of this remark is given in Section~\ref{SOmega2}.
In the case of Assumption~4',  {a partial invariance principle theorem} may be required in the calculation of
Remark~\ref{rmkwij}, which only  replaces   $\mathbf U_2^*\mathbf X,\mathbf V_2^*\mathbf Y$ by  $\mathbf U_2^*\mathbf W$ and $\mathbf V_2^*\mathbf Z$ as column to column, respectively,  but keeps $\mathbf U_1^*\mathbf X,\mathbf V_1^*\mathbf Y$ unchanged.
Due to space constraints, we opt to omit the details of the
partial invariance principle theorem  here, and these can be further refined in future
work. This remark is used in the simulations of Case I under
non-Gaussian assumptions.


\section{CLT for generalized spiked Fisher matrices}  
 \label{SecCLT}

We employ the invariance principle to show the CLT for the  spiked eigenvalues  of
a generalized spiked Fisher matrix $\mathbf F$.
As mentioned in Proposition~\ref{PT}, a packet of $m_k$ consecutive sample
eigenvalues $\{l_{p,j}(\mathbf F), j \in \mathcal{J}_k\}$ converge to a limit $\rho_k$ laying
outside the support of the limiting spectral distribution (LSD) 
of $\mathbf F$.
To improve the CLT in \cite{WangYao2017}, we consider a more general spiked
Fisher matrix, $\bf F$, in (\ref{F}) and the renormalized random vector.
\be
\gamma_k=(\gamma_{kj}, {j \in \mathcal{J}_k}):=~\left(\sqrt{p} \big\{\frac{l_{p,j}(\mathbf F)}{\psi_n(\alpha_k)}-1\big\} , j \in \mathcal{J}_k \right).\label{mrv}
\ee
Then, the CLT for the  renormalized random vector   
$\gamma_k$   is provided in the following  theorem. 

\begin{theorem}
Suppose that  Assumptions~1--5  
hold. For each distinct  spiked eigenvalue
$\alpha_k$ (i.e,  $\psi'(\alpha_k)>0$, \cite{Jiangetal2019}) 
with multiplicity $m_k$, the $m_k$-dimensional real vector
 $\gamma_k$ defined in
(\ref{mrv})
converges  weakly to the joint  distribution  of the $m_k$ eigenvalues of Gaussian random matrix  $-\left[\boldsymbol\Omega_{\psi_{k}}\right]_{kk}/{\phi_k}$,
where
$
\phi_k =1+c_{2}\psi^2_{k}m_2(\psi_{k}) +2 c_2\psi_{k}m(\psi_{k})+\alpha_k\psi_{k}\um_2(\psi_{k})+\alpha_k \um(\psi_{k})
$
and  $\psi_k^2m_2(\psi_k)$ is the limit of $\psi_{n,k}^2m_2(\psi_{n,k})$ even if $\alpha_k\to\infty$.
Furthermore, $\boldsymbol\Omega_{\psi_{k}}$ is defined in Corollary~{\ref{coro1}}  and
$\left[\boldsymbol\Omega_{\psi_{k}}\right]_{kk}$ is the $k$th diagonal block of  $\boldsymbol\Omega_{\psi_{k}}$ corresponding to the indices $\{i,j \in \mathcal{J}_k\}$.
\label{CLT}
\end{theorem}

\begin{proof}
As shown in Section~\ref{Pre}, every sample spiked eigenvalue of $\mathbf F$, $l_{p,j}, j \in \mathcal{J}_k, k=1,\cdots, K$, 
satisfies the equation (\ref{eigeneq8}). 
Furthermore,
since $\psi_{n,k}$ satisfies the equation 
(\ref{0eqa}),
it means that the population spiked eigenvalues $\alpha_u$ in the $u$th diagonal block of $\mathbf D_1$ makes
$\psi_{n,k}+c_{2} \psi^2_{n,k}m(\psi_{n,k}) + \psi_{n,k}\um(\psi_{n,k}) \alpha_u$
keep away from 0, if $u\neq k$;
and satisfies
$\psi_{n,k}+c_{2}\psi^2_{n,k}m(\psi_{n,k})+\psi_{n,k}\um(\psi_{n,k}) \alpha_k=0$.  For non-zero limit of spiked eigenvalue, $\psi_{n,k}$, each $k$th diagonal block of the  equation (\ref{eigeneq8}) is  multiplied $p^{{1}/{4}}$  by rows and columns, respectively. Then, by Lemma~4.1 in  
\cite{BaiMiaoRao1991},
it follows from equation (\ref{eigeneq8}) that
\begin{align}
&\bigg|\gamma_{kj}
\psi_{n,k}\!\big\{\! 1\!\!+\!\!c_{2}\psi^2_{n,k}m_2(\psi_{n,k}) \!\!+\!\!2 c_2\psi_{n,k}m(\psi_{n,k})\!\!+\!\!\alpha_k\psi_{n,k}\um_2(\psi_{n,k})\!\!+\!\!\alpha_k \um(\psi_{n,k})\!\big\}\!\mathbf I_{m_k}\label{clteq}\\
&\!+\!\psi_{n,k}\big[\boldsymbol\Omega_M(\psi_{n,k},\mathbf X,\mathbf Y)\big]_{kk}+o(\psi_{n,k}) \bigg|=0,\nonumber
\end{align}
where $\left[~\cdot ~\right]_{kk}$ is  the $k$th diagonal block of  a matrix  corresponding to the indices $\{i,j \in \mathcal{J}_k\}$.
According to the  Skorokhod strong representation in 
 \cite{Skorokhod1956, HuBai2014},
it follows  that the  convergence
of $\boldsymbol\Omega_M(\psi_{n,k},\mathbf X,\mathbf Y)$  and (\ref{eigeneq8}) can be achieved simultaneously in probability 1 by choosing an appropriate probability space.

Let
$
\phi_k =1+c_{2}\psi^2_{n,k}m_2(\psi_{n,k}) +2 c_2\psi_{n,k}m(\psi_{n,k})+\alpha_k\psi_{n,k}\um_2(\psi_{n,k})+\alpha_k \um(\psi_{n,k}). 
$
The equation (\ref{clteq}) arrives at
$\Big|\gamma_{kj} \phi_k  \mathbf I_{m_k}
+\big[\boldsymbol\Omega_{\psi_{k}}\big]_{kk}+o(1)
 \Big|  =0.$
Thus, it is obvious that, $\gamma_{kj}$ asymptotically satisfies the following equation
\begin{equation}
\Big|\gamma_{kj} \cdot \phi_k  \mathbf I_{m_k}
+\big[\boldsymbol\Omega_{\psi_{k}}\big]_{kk}
 \Big|  =0,\label{fineq0}
 \end{equation}
 where $\boldsymbol\Omega_{\psi_{k}}$ is an $M\times M$ Hermitian matrix, being  the limiting distribution of
 $\boldsymbol\Omega_M(\psi_{n,k},\mathbf X,\mathbf Y)$.

Therefore,  by the equation (\ref{fineq0}), the $m_k$-dimensional real vector
$\{\gamma_{kj}, j\in \mathcal{J}_k\}$ converges  weakly to the distribution  of  the $m_k$ eigenvalues of the Gaussian random matrix
$-\left[\boldsymbol\Omega_{\psi_{k}}\right]_{kk}\big/{\phi_k} $
for each distinct spiked eigenvalue. The   distribution of $\boldsymbol\Omega_{\psi_{k}}$ is  detailed in Corollary~\ref{coro1}.
Then, the CLT for each distinct spiked eigenvalue of a
generalized spiked Fisher matrix is established.
\end{proof}

\begin{remark}
Suppose that  $\mathbf X, \mathbf Y$ satisfy Assumptions~1,2,3 and 5,
but Assumption~4 is weakened to  Assumption~4' in Remark~\ref{rmkwij}. Then
all the conclusions of Theorem~\ref{CLT} still hold, but the limiting distribution of
$\boldsymbol\Omega_M(\psi_{n,k}, \mathbf X,\mathbf Y)$ tends to
an $M\times M$ Hermitian Gaussian matrix $\boldsymbol\Omega_{\phi_{k}}=(\omega_{st})$ whose variances and covariances are
defined in Remark~\ref{rmkwij}.\\[-5mm]
\label{rmkclt}
\end{remark}

\section{Applications}\label{sec5.1}


{\color{blue}In this section, we present two applications of  Theorem~\ref{CLT} in the  linear regression model and  statistical signal processing.
First one is to derive the   theoretical power of  Roy Maximum Root test in linear regression model,
the other is applied to the signal detection in wireless communication.}
\subsection{Linear regression model}\label{sec5.1.1}

Consider  a $p$-dimensional linear regression model
\begin{equation} \label{linear model}
  \mathbf w_i=\mathbf B\mathbf z_i+\bvare_i,   i=1,\ldots,n,
\end{equation}
where $\bvare_i,$ $i=1,\cdots, n$ is a sequence of independent and identically distributed normal
error vector $\mathcal{N}_p(0,\boldsymbol\Sigma)$, $\mathbf B$ is  a
$p\times  q_0$ regression matrix, and $(\bz_i)_{i,\cdots,n}$ a sequence of known
regression variables of dimension $q_0$.
In this section, assume that $n\geq p+q_0$ and the rank of $\mathbf Z=
(\mathbf z_1, \cdots, \mathbf z_n)$ is $q_0$.

Define a block decomposition
 $\mathbf B= (\mathbf B_1, \mathbf B_2)$ with
$ q_1$ and $q_2$ columns, respectively ($q_0=q_1+q_2$).
Partition the
regression variables $\{\mathbf z_i\}$ accordingly as in
$\mathbf z_i=\left( \mathbf z_{i1}', \mathbf z_{i2}' \right)'$.

Consider to test hypothesis that
\begin{equation}
 \mathcal{H}_0 : ~\mathbf B_1  =  \mathbf B_{1}^0 \quad {\rm v.s. }  \quad  \mathcal{H}_1 : ~\mathbf B_1  \neq  \mathbf B_{1}^0~,
\label{HB0}
\end{equation}
where  $\mathbf B_1^0$ is a known matrix.
Roy in \cite{Roy1953} proposed $\lambda_{1}$, the maximum eigenvalue of $\mathbf H\mathbf G^{-1}$, to test on linear regression hypothesis (\ref{HB0}), where
\[\mathbf G=n\hat{\boldsymbol\Sigma}=\sum\limits_{i=1}^n
  (\mathbf w_i-\widehat{\mathbf B}\mathbf z_i)(\mathbf w_i-\widehat{\mathbf B}\mathbf z_i)'; \quad \mathbf H=(\widehat{\mathbf B}_1-\mathbf B_1^0)
\mathbf A_{11:2}(\widehat{\mathbf B}_1-\mathbf B_1^0)',\]
$\widehat{\mathbf B}$ is the maximum likelihood estimators of $\mathbf B$,
$\widehat{\mathbf B}_1$   denotes as the former $q_1$ columns of $\widehat{\mathbf B}$, and
  $\mathbf A_{11:2}= \sum_{i=1}^n \mathbf z_{i1} \mathbf z_{i1}'-\sum_{i=1}^n \mathbf z_{i1} \mathbf z_{i2}' \big( \sum_{i=1}^n \mathbf z_{i2} \mathbf z_{i2}'\big)^{-1} \sum_{i=1}^n \mathbf z_{i2} \mathbf z_{i1}'$.
The distribution of $\lambda_{1}$ can be obtained from the joint density  by the integration over the supporting set of all eigenvalues.
Roy in \cite{Roy1945} developed a method of integration  and gave the distribution of $\lambda_1/(1+\lambda_1)$ when $p=2$.
However, the integration
 is more difficult  than that for the density of the roots of $\mathbf H\mathbf G^{-1}$ with the
 increasing the dimensionality.
 By Lemmas~8.4.1 and 8.4.2 in \cite{Anderson2003},
 it is known that
$
\mathbf G \sim W_p(\boldsymbol\Sigma,n-q_0),~
\mathbf H \sim
W_p(\boldsymbol\Sigma,q_1),
$
and they are independent on each other under the Gaussian assumption.
So
 $(n-q_0)q_1^{-1}\mathbf H\mathbf G^{-1}$  can be viewed as a generalized spiked Fisher matrix.
Consider the large-dimensional setting
 \be
 \tilde c_{n_1}=p/q_1 \rightarrow \tilde c_1 \in (0, \infty),  ~\tilde c_{n_2}=p/(n-q_0) \rightarrow  \tilde c_2 \in (0, 1). \label{limitingScheme}
 \ee
{\color{black}
 As shown in  \cite{Hanetal2016}, the  largest root of $(n-q_0)q_1^{-1}\mathbf H\mathbf G^{-1}$,  $l_{p,1}=(n-q_0)q_1^{-1}\lambda_1$, follows the
 Tracy-Widom law under the null hypothesis.
 Its rejection region at the 0.05 significance level is
 \be\label{reject}
 \{l_{p,1}> \psi_0+\sigma_{\!tw}C_{0.95}\},
 \ee
where $h^2=\tilde c_1+\tilde c_2-\tilde c_1\tilde c_2$, $\psi_0=(1+h)^2(1-\tilde c_2)^{-2}$ is the limit of the largest root under the null hypothesis and
$C_{0.95}$ is the 95th percentile of the Tracy-Widom distribution.

The value of $\sigma_{\!tw}$ is given by several trigonometric equations in
\cite{Hanetal2016}. In order to give a simpler expression, \cite{WangYao2017}
provided a result according to \cite{Hanetal2016}.  Their result is only
related to one of sample sizes  and the  dimensionality. Thus, we compared the results of
the two and found that they were not the same. Therefore, we recompute the value of $\sigma_{\!tw}$ according to \cite{Hanetal2016} and provide a simpler expression as follows, which is identical with the one in \cite{Hanetal2016}, i.e.
\begin{align*}
&\sigma_{\!tw}^3=
\frac{\tilde c_1^2(\tilde c_1+h)^4(\tilde c_1+\tilde c_2)^6}
{(n-q_0+q_1)^2 h \tilde c_2^2\big\{(\tilde c_1+\tilde c_2)^2-\tilde c_2(\tilde c_1+h)^2\big\}^4}.
\end{align*}
}

Based on the  rejection  region (\ref{reject}),
 Theorem~\ref{CLT}  is applied to derive the asymptotic distribution of $l_{p,1}$ and provide the power function under
 the alternative hypothesis,  which is detailed  as below.
 \begin{theorem}\label{T4.1}
 For testing hypothesis (\ref{HB0}), if  the large-dimensional limiting scheme
 (\ref{limitingScheme}) holds, then the asymptotic distribution of the largest
 sample eigenvalue of $(n-q_0)q_1^{-1}\bH\bG^{-1}$, $l_{p,1}$, is
 \be
\Lambda_1=\sqrt{p}\Big(\frac{l_{p,1}}{\psi_{n,1}}-1\Big) \Big/ \sigma_1 \Rightarrow \mathcal{N}(0, 1), ~~ \text{under~} \mathcal{H}_1,
\label{AD}
\ee
 where $\sigma_1^2=2\theta_1/\phi_k^2$ for the general real case with Assumption~4 
 and  $\sigma_1^2=(2\theta_1+\beta_{x,iiii}\nu_1+\beta_{y,iiii}\nu_2 )/\phi_k^2$
for the  real case with Assumption~4' instead.
 Then, the power of Roy Maximum Root test on linear regression hypothesis is calculated by

\begin{align}
power=\Phi \bigg(-\frac{\sqrt{p}\big\{\psi_0+ \sigma_{\!tw}C_{0.95}-\psi_{n,1}\big\}}{\psi_{n,1}\sigma_1}\bigg),
\end{align}
where $ C_{0.95}$ is 95\% quantile of Tracy-Widom distribution and $\Phi$ is standard Gaussian distribution.
\end{theorem}

In practice,  $\psi_{n,1}, \sigma_1$ in the asymptotic
distribution (\ref{AD}) is  involved with the unknown population
largest  spike $\alpha_1$ and the LSD of
$(n-q_0)q_1^{-1}\bH\bG^{-1}$, so we
 provide some estimations to calculate the estimated test statistic
 $\hat \Lambda_1$ instead of $ \Lambda_1$ in  (\ref{AD}) as follows.

First, for the population spike $\alpha_1$,  it follows  by the
first equation in (\ref{0eqa}) that the equation
\[ 1+\tilde c_2 l_{p,1} m(l_{p,1})+\um(l_{p,1}) \alpha_1=0\]
holds  approximately.
So we get the estimation of $\alpha_1$,
\[
\hat \alpha_1= -\displaystyle\frac{1+\tilde c_2 l_{p,1} m(l_{p,1})}{\um(l_{p,1})}, 
\]
where $m(l_{p,1})$ and  $\um(l_{p,1})$ can be estimated by
\[
\hat m(l_{p,1})=\frac{1}{p-|\mathcal{J}_1|}\sum\limits_{i\notin \mathcal{J}_1}(l_{p,i}-l_{p,1})^{-1} \quad \text{and} \quad \hat \um (l_{p,1})=-\displaystyle\frac{1-\tilde c_1}{l_{p,1}}+\tilde c_1 \hat m(l_{p,1}), 
\]
respectively, with $r_{i}$ and $\mathcal{J}_1$ defined as
  $r_{i}=|l_{p,i}-l_{p,1}|/|l_{p,1}|$ and  $\mathcal{J}_1=\{i\in (1,\cdots, p): r_{i}\leq 0.2 \}$.  The set $\mathcal{J}_1$ is  selected to avoid  the effect of multiple roots and make the estimator more accurate.
  The constant 0.2 is a more suitable threshold value according to our simulated results.
Moreover, the following estimators  may be used  to calculate
$\psi_{n,1}$ and $\sigma_1$.
\begin{align*}
&\hat\psi_{n,1}=\psi(\hat \alpha_1);\\
& \hat m(\psi_{n,1})=\frac{1}{p}\sum\limits_{ i\in \mathcal{J}_1}(l_{p,i}-\hat \psi_{n,1})^{-1};~ \hat \um(\psi_{n,1})=-\displaystyle\frac{1-\tilde c_1}{\hat\psi_{n,1}}+\tilde c_1 \hat m(\hat\psi_{n,1});\non[-2mm]
&\hat m_2(\psi_{n,1})=\frac{1}{p}\sum\limits_{ i\in \mathcal{J}_1}(l_{p,i}-\hat\psi_{n,1})^{-2};~
\hat \um_2(\psi_{n,1})=\displaystyle\frac{1-\tilde c_1}{(\hat\psi_{n,1})^2}+\tilde c_1 \hat m_2(\hat\psi_{n,1});\non[-2mm]
&\hat m_3(\psi_{n,1})=\frac{1}{p}\sum\limits_{ i\in \mathcal{J}_1}l_{p,i}(l_{p,i}-\hat\psi_{n,1})^{-2};
\end{align*}
where the integral over the $H(t)$ in (\ref{psik}) 
 can be estimated by the empirical spectral distribution.
Thus, the $\Lambda_1$ in (\ref{AD})  
 can be estimated by the above estimators.

{\color{blue}
\subsection{Signal detection}\label{app2}

 The literature \cite{Hanetal2016} established the Tracy-Widom law for the largest eigenvalue of the Fisher matrix and applied the results to the signal detection problem.
  In signal detection or cognitive, the model generally has the following form:
 \begin{align}
 {\mathbf y}_t={\mathbf A}{\mathbf x}_t + {\mathbf \Sigma}^{1/2}{\boldsymbol e}_t, \quad t=1,2,\cdots,m,
 \end{align}
 where ${\mathbf y}_t$ is a $p$-dimensional observations, ${\mathbf A}$ is a $p\times k$ mixing matrix,  ${\mathbf x}_t$ is a $k\times 1$ low-dimensional signal with covariance matrix ${\bI}_k$  while ${\boldsymbol e}_t$ is an i.i.d noise with covariance matrix ${\bI}_p$. The signal ${\mathbf x}_t$ is independent with the noise ${\boldsymbol e}_t$. For more details, see \cite{ZL2009,Levanon1988,NS2010}. A fundamental task in signal processing is to test
 \begin{align}
 \mathcal{H}_0 : ~\bf A  =  \bf 0 \quad {\rm v.s. }  \quad  \mathcal{H}_1 : ~\bA  \neq  \bf 0.
 \label{HA}
 \end{align}
 In engineering, one can have additional independent noise-only observations ${\mathbf z}_t={\mathbf \Sigma}^{1/2}{\mathbf r}_t,  t=1,\cdots,T$. Let

 \begin{align}
 {\mathbf Y}=({\mathbf y}_1,{\mathbf y}_2,\cdots,{\mathbf y}_m),\quad {\mathbf Z}=({\mathbf z}_1,{\mathbf z}_2,\cdots,{\mathbf z}_T).
 \end{align}
 We define the Fisher matrix
 \begin{align}
 {\mathbf F}=\frac{T}{m}({\mathbf Z}{\mathbf Z}^*)^{-1}({\mathbf Y}{\mathbf Y}^*),
 \end{align}
 and use the symbols $l_1$ and $\beta_1$ to denote the largest eigenvalue of ${\mathbf F}$ and ${\mathbf \Sigma}^{-1}({\mathbf A}{\mathbf A}^*+{\mathbf \Sigma})$, respectively. We use the statistic $l_1$ to test the hypothesis (\ref{HA}). According to \cite{Hanetal2016}, $l_1$ after scaling tends to Tracy-Widom law under the null hypothesis $\bf A  =  \bf 0$. By (\ref{AD}), we conclude the theoretical power for correlated noise detection as

 \begin{align}\label{spower}
 P_R(\beta_1)=\Phi \bigg(-\frac{\sqrt{p}\big\{\psi_0+ \sigma_{\!tw}C_{0.95}-\psi_{n,1}(\beta_1)\big\}}{\psi_{n,1}(\beta_1)\sigma_1}\bigg),
 \end{align}
 where the notations defined similar to Theorem \ref{T4.1}. It is worth pointing out that Theorem 7.1 in  \cite{WangYao2017} is a special case of (\ref{spower}).

}

%
%
%



\section{Simulation Study}\label{Sim}
We first conduct simulation to compare the performance of the
limiting distribution for the two-sample spiked model with the one
derived in \cite{WangYao2017}.

\subsection{Simulations for Section~\ref{SecCLT}}
We consider two scenarios:
\begin{description}
\item[ Case I:] The matrix $\bT_p^*\bT_p$ is taken to be  a finite-rank perturbation of
an identity matrix $\bI_p$, where $\boldsymbol\Sigma_2=\bI_p$ and $\boldsymbol\Sigma_1$ is an identity matrix with  the spikes $(20,0.2,0.1)$ of the multiplicity $(1,2,1)$ in the descending order and  thus $K=3$ and $M=4$ as proposed in  \cite{WangYao2017}.
\item[Case II:]  The matrix $\bT_p^*\bT_p$ is a general positive definite matrix,
but not necessary  with diagonal blocks independence assumption. 
It is designed as below:
$\boldsymbol\Sigma_2=\bI_p$ and $\boldsymbol\Sigma_1=\bU_0 \boldsymbol\Lambda \bU_0^*$, where  $\boldsymbol\Lambda$ is a diagonal matrix  consisting of
the spikes   $(20,0.2,0.1)$ with multiplicity $(1,2,1)$ and the other eigenvalues being 1 in the descending order.
Let $\bU_0$  be equal to the matrix composed of eigenvectors
 of the $p\times p$
matrix $(\rho^{|i-j|})_{i,j=1,\ldots,p}$
with $\rho=0.5$.
\end{description}

For each scenario, we first consider two populations as following: In the
first population, $x_{ij}$ and $y_{ij}$  are  both i.i.d. samples from $N(0,1)$.
In the second population, $x_{ij}$ and $y_{ij}$ are
i.i.d. samples from $P\{x_{ij}=\pm 1\}=P\{y_{ij}=\pm 1\}=1/2$.
Thus, $\rE|x_{ij}|^4=\rE|y_{ij}|^4=1$. This aims to illustrate the invariance principle of
large-dimensional spiked Fisher matrices.

To further demonstrate the CLT derived in Section 3 is valid for a distribution with infinite fourth moments
under Assumption~4, we generate i.i.d. samples $x_{ij}$ and $y_{ij}$  from
$2^{-1/2}t(4)$ population distribution under the setting of Case II,
hence, $\rE x_{ij}=\rE y_{ij}=0$, $\rE x_{ij}^2=\rE y_{ij}^2=1$, while the fourth moments of
$x_{ij}, y_{ij}$ are infinite. Since Assumptions~4' for the CLT derived in Section 3  is not met in the distribution
$t(4)$ under the setting of Case I, then the new CLT does not hold for $t(4)$ with Case I.
Furthermore, the limiting distribution for the two-sample spiked model
derived in \cite{WangYao2017} is not applicable for $t(4)$ either. Thus, we examine the performance
of the newly derived limiting distribution only.
In this simulation, we set $p=200$, $ n_1=1000$ and $n_2=400$, and we
conduct 1000 replications for each case.

\begin{table}
\caption{KS Statistic and Percentiles of asymptotical distributions  of the standardized $\hat{\gamma}_1$,
$\hat{\gamma}_2^*=(\hat{\gamma}_{21}+\hat{\gamma}_{22})/2$,
 and $\hat{\gamma}_3$ derived by our new method and Wang and Yao \cite{WangYao2017}}
\label{tab1}
\begin{tabularx}{14.5cm}{XXXcccccXXXXX}
\hline
E.V. & Case & Method & 1\% & 5\% & 10\% & 25\% & 50\% & 75\% & 90\% & 95\% & 99\% & KS\\
 & \multicolumn{2}{c} {Limiting $N(0,1)$} & -2.326 &-1.645 &-1.282 & -0.674 & 0 & 0.674 &1.282 &1.645 &2.326& $-$\\
\hline
&& \multicolumn{9}{c} {$x_{ij}\sim N(0,1)$ and $y_{ij}\sim N(0,1)$}\\
\hline
$\gamma_1$ & I  & New & -2.005 &-1.455 &-1.175 & -0.650 & -0.043 & 0.680 &1.400 &1.791 &2.606&0.025\\
 & & WY &  -2.047 &-1.487 &-1.200 & -0.665 & -0.045 & 0.694 &1.429 &1.828 &2.661&0.028\\
$\gamma_1$ & II  & New &  -1.996 &-1.540 &-1.191 & -0.658 & -0.009 & 0.671 &1.378 &1.775 &2.660&0.031\\
  & & WY & -1.975 &-1.524 &-1.178 & -0.652 & -0.009 & 0.663 &1.362 &1.755 &2.630&0.030\\
$\gamma^*_{2}$ & I  & New &  -1.493 &-1.017 &-0.779 & -0.257 & 0.265 & 0.814 &1.330 &1.631 &2.233&0.150\\
 & & WY & -1.500 &-1.005 &-0.762 & -0.226 & 0.313 & 0.880 &1.412&1.718 &2.351&0.162 \\
$\gamma^*_{2}$ &II  & New & -1.574 &-1.065 &-0.761 &-0.301&0.307 & 0.857& 1.301 &1.682 &2.386 &0.137  \\
  & & WY &  -1.456 &-0.925 &-0.624 &-0.160&0.452 & 1.007& 1.459 &1.838 &2.555 &0.194 \\
$\gamma_{3}$ & I  & New & -1.930&-1.502 &-1.183 & -0.609& -0.016 & 0.686 &1.302 &1.685 &2.769&0.022\\
 & & WY &  -1.802 &-1.36 &-1.038 & -0.455 & 0.147 & 0.862 &1.487 &1.878 &2.979&0.075\\
$\gamma_3$ & II  & New &  -1.954 &-1.554 &-1.253 & -0.703 & -0.052 & 0.562 &1.167 &1.609 &2.398&0.050\\
 & & WY &  -1.963 &-1.562 &-1.259 & -0.708 & -0.055 & 0.561 &1.167 &1.611 &2.405&0.051\\
\hline
 & &    \multicolumn{9}{c} {$P(x_{ij}=\pm 1)=1/2$ and $P(y_{ij}=\pm 1)=1/2$ }\\
\hline
$\gamma_1$ & I  & New &  -1.957 &-1.518 &-1.240 & -0.646 & -0.026 & 0.681 &1.363 &1.753 &2.694&0.025\\
 & & WY & -1.968 &-1.516 &-1.152 & -0.640 & -0.028 & 0.694 &1.391 &1.788 &2.749&0.026\\
$\gamma_1$ & II  & New & -2.019 &-1.484 &-1.187 & -0.648 & -0.007 & 0.637 &1.410 &1.823 &2.503&0.023\\
 & & WY &  -2.883 &-2.119 &-1.694 & -0.925 & -0.011 & 0.909 &2.011 &2.599 &3.572&0.093\\
$\gamma^*_{2}$ &I  & New & -1.229 &-0.812 &-0.554 & -0.029 & -0.509 &1.096&1.619 &1.925 &2.506&0.250 \\
  & & WY &  -0.378 &-0.021 &0.169 & 0.556 & 0.952 & 1.384 &1.768 &1.995 &2.430&0.491 \\
$\gamma^*_{2}$&II  & New &  -1.493 &-0.956 &-0.680 &-0.178&0.403 & 0.918& 1.476 &1.798 &2.438 &0.185 \\
  & & WY & -1.876 &-1.132&-0.757 &-0.072&0.721 & 1.417& 2.180 &2.629 &3.496 &0.267 \\
$\gamma_{3}$ &  I  & New &  -2.381 &-1.646 &-1.296 & -0.695 & -0.013 & 0.578 &1.114 &1.520 &2.217&0.047\\
 & & WY & -2.048 &-1.272 &-0.899 & -0.268 & -0.450 & 1.073 &1.639 &2.069 &2.802&0.175\\
$\gamma_3$ & II  & New &  -2.108 &-1.527 &-1.215 & -0.666 & 0.017 & 0.711 &1.357 &1.733 &2.615&0.026\\
 & & WY & -5.907 &-4.283 &-3.407 & -1.870 & 0.056 & 1.982 &3.789 &4.842 &7.308&0.236\\
\hline
 & &    \multicolumn{9}{c} {$x_{ij}\sim 2^{-1/2}t(4)$ and $y_{ij}\sim2^{-1/2}t(4)$}\\
\hline
$\gamma_1$ &II  & New & -1.673 &-1.312 &-0.969 & -0.418 & 0.253 & 0.939 &1.703 &2.199 &3.324&0.112\\
$\gamma^*_{2}$ &II  & New &  -3.936 &-1.894 &-1.343 & -0.703 & -0.029 & 0.517 &1.004 &1.324 &1.811&0.068 \\
$\gamma_3$ & II  & New &  -2.873 &-1.959 &-1.512 & -0.981 & -0.274 & 0.444 &1.112 &1.510 &2.249&0.118\\
\hline
\end{tabularx}
\end{table}

\begin{figure}[htbp]
\begin{center}
\includegraphics[width = .31\textwidth]{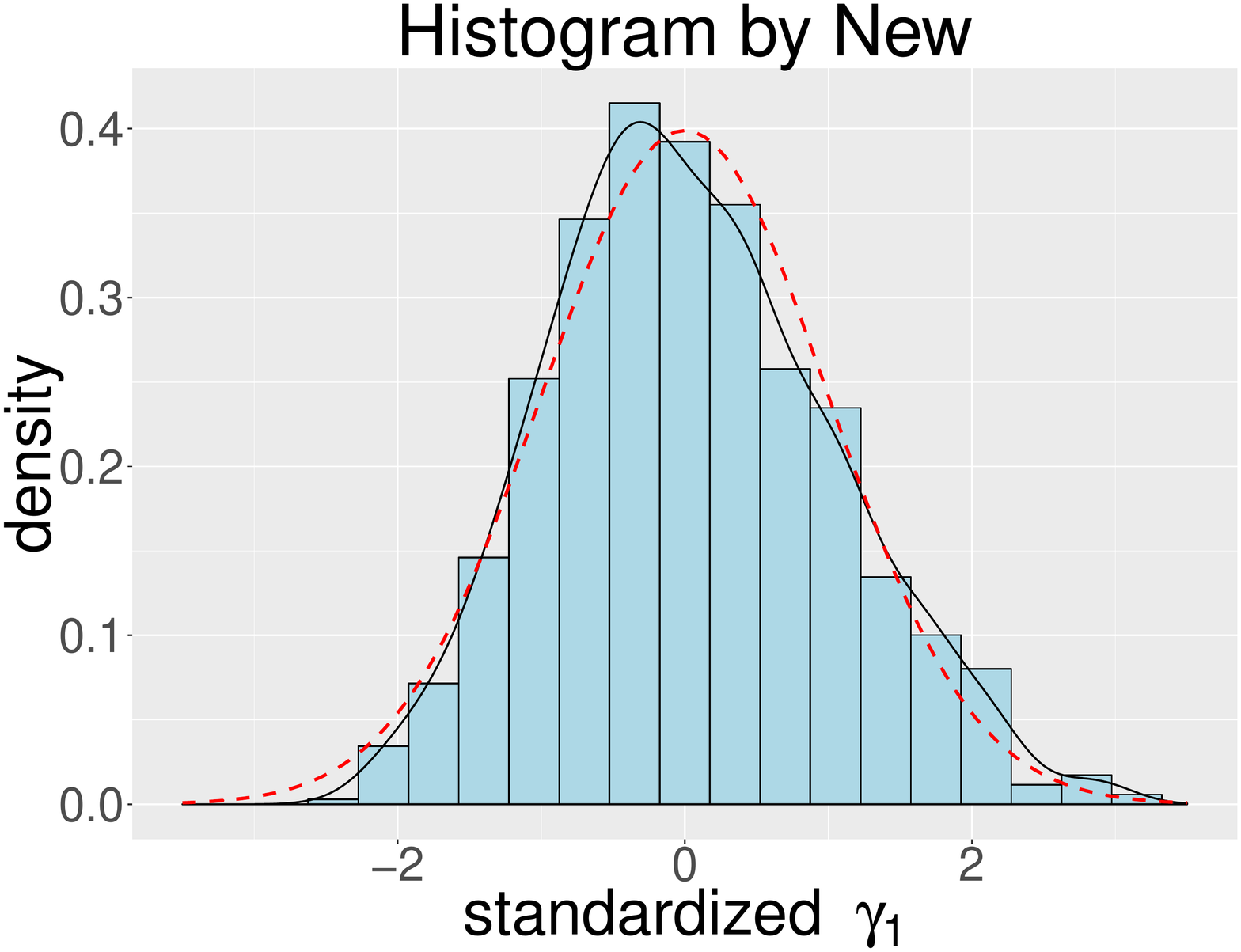}\quad
\includegraphics[width = .31\textwidth] {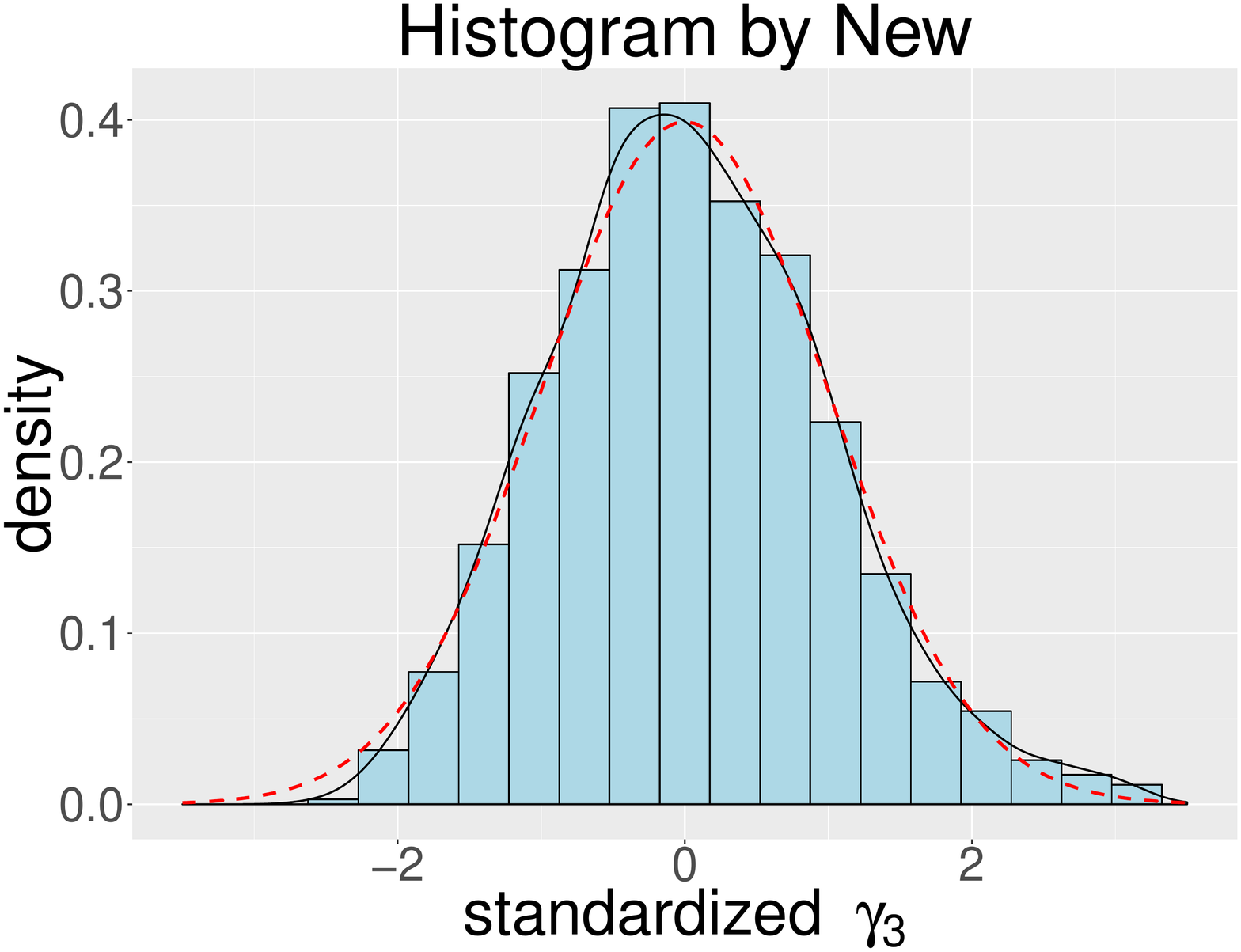}\quad
\includegraphics[width = .31\textwidth] {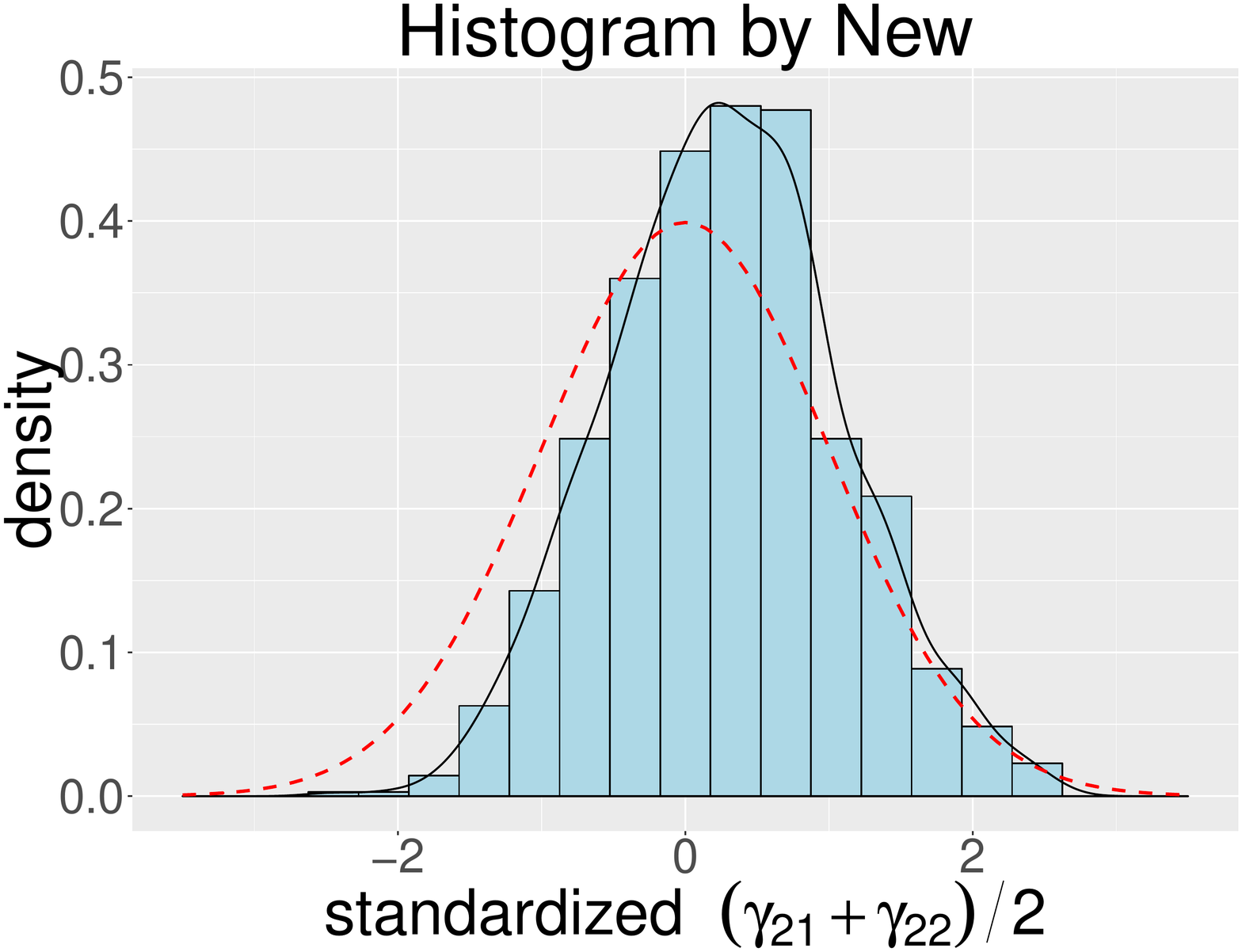}\\
\includegraphics[width = .31\textwidth]{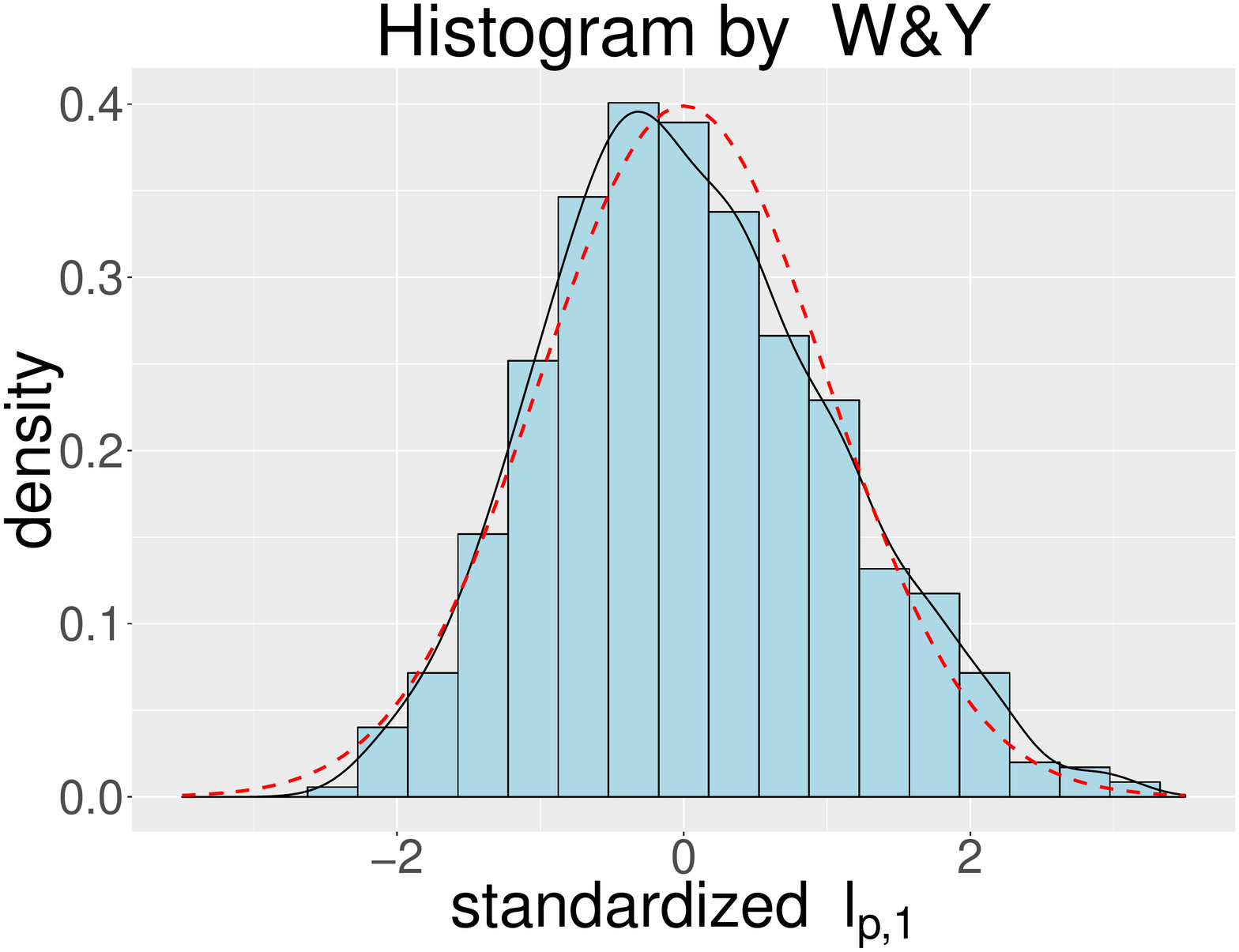}\quad
\includegraphics[width = .31\textwidth] {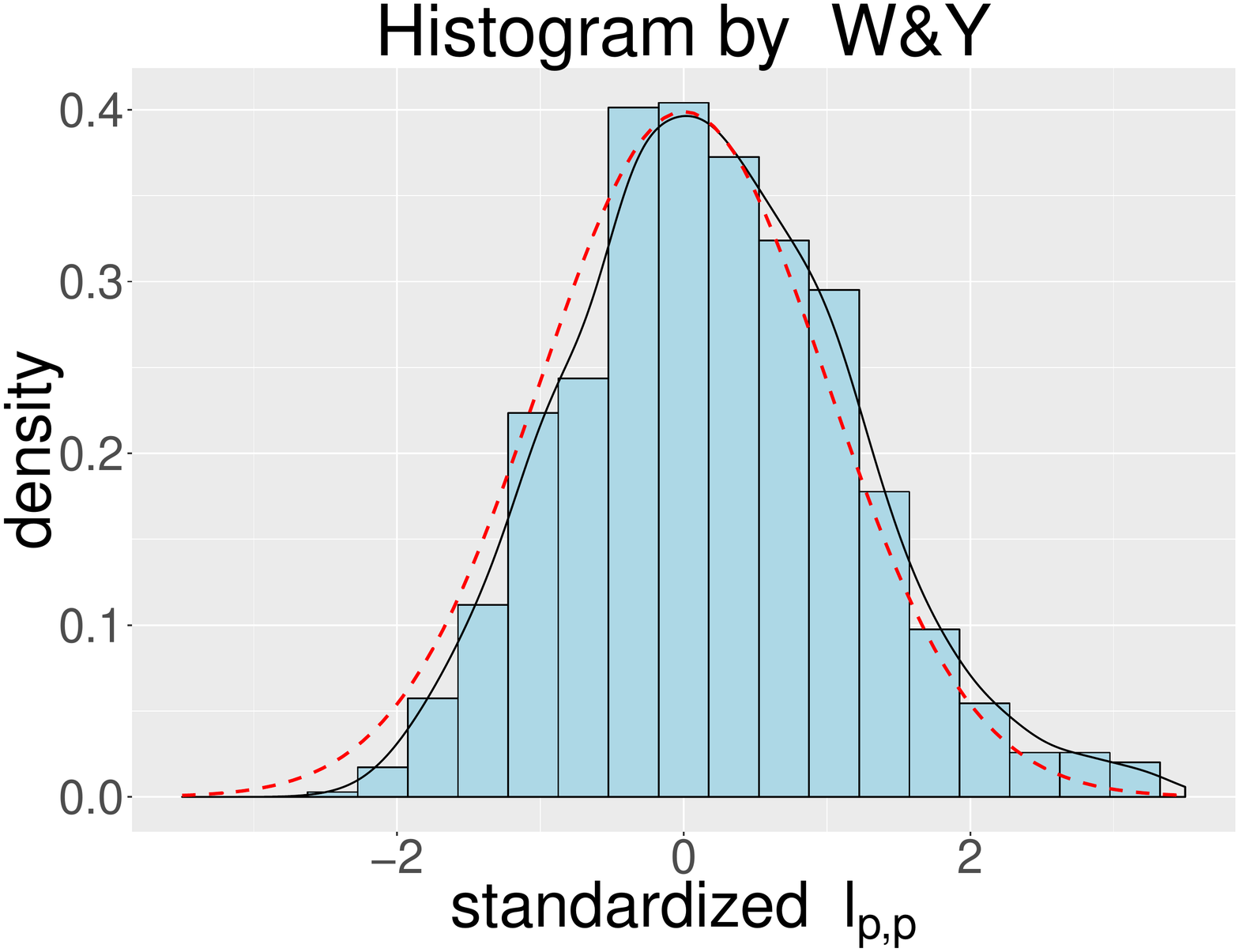}\quad
\includegraphics[width = .31\textwidth] {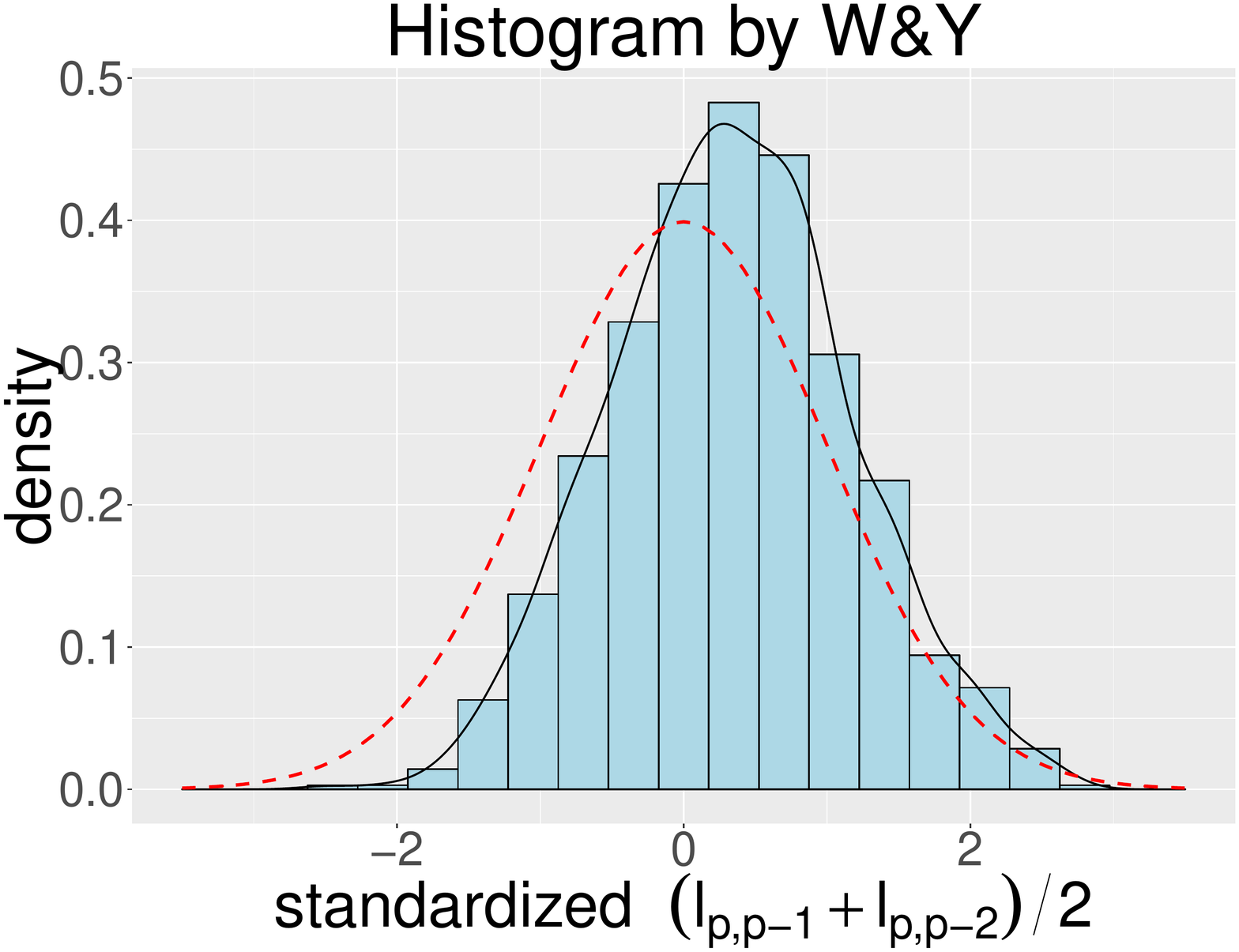}
\end{center}
\centerline{Case I  }
\begin{center}
\includegraphics[width = .31\textwidth]{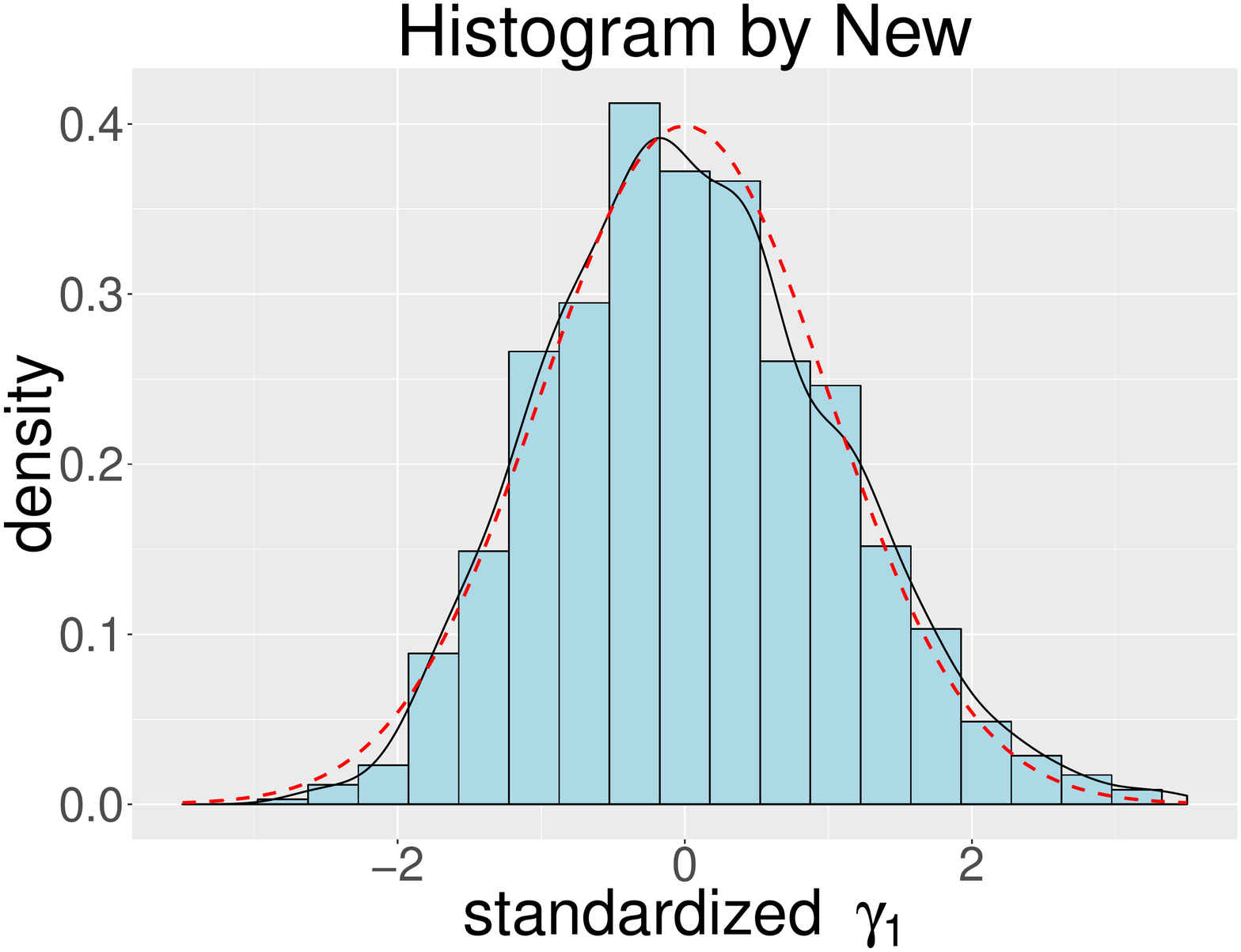}\quad
\includegraphics[width = .31\textwidth]{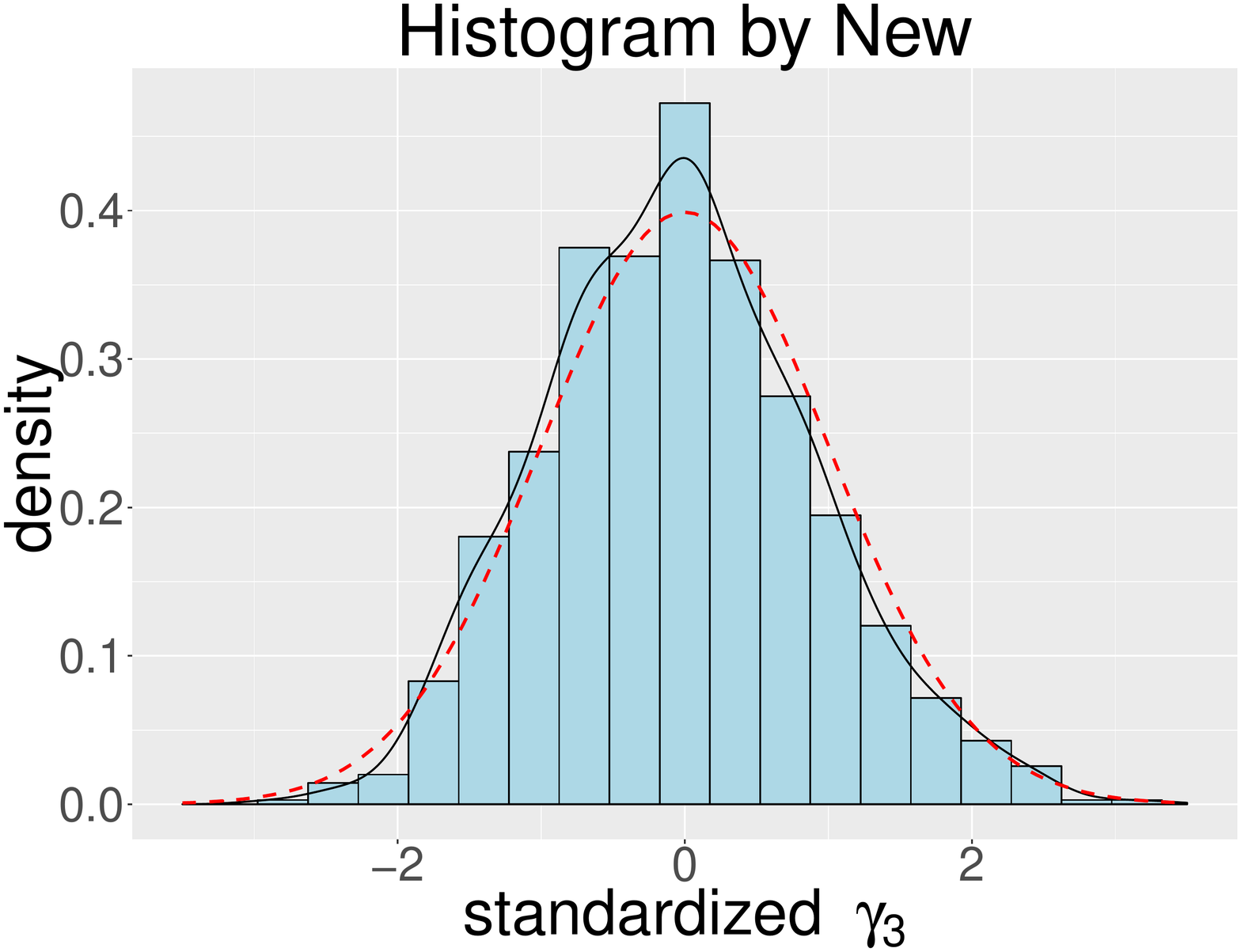}\quad
\includegraphics[width = .31\textwidth]{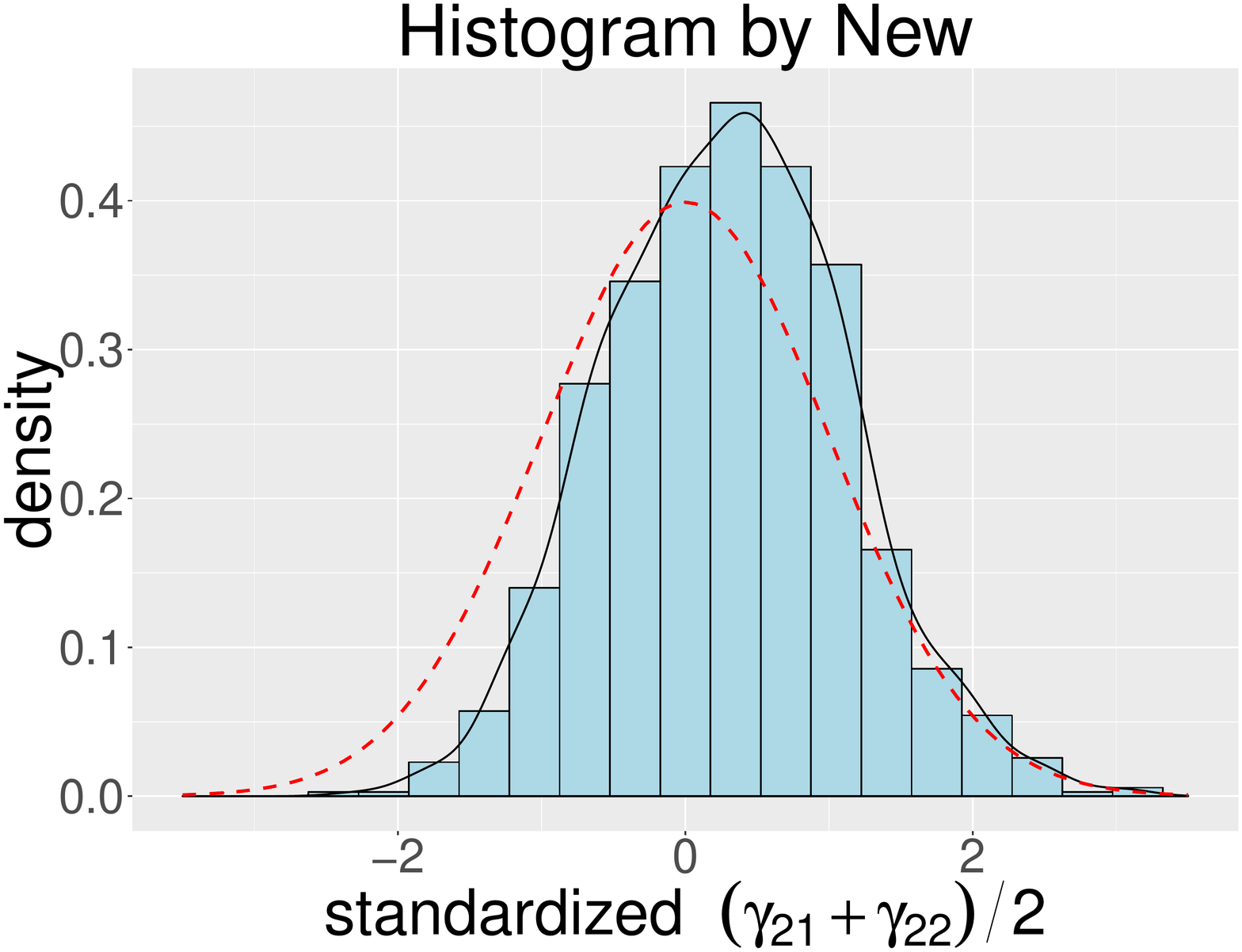}\\
\includegraphics[width = .31\textwidth]{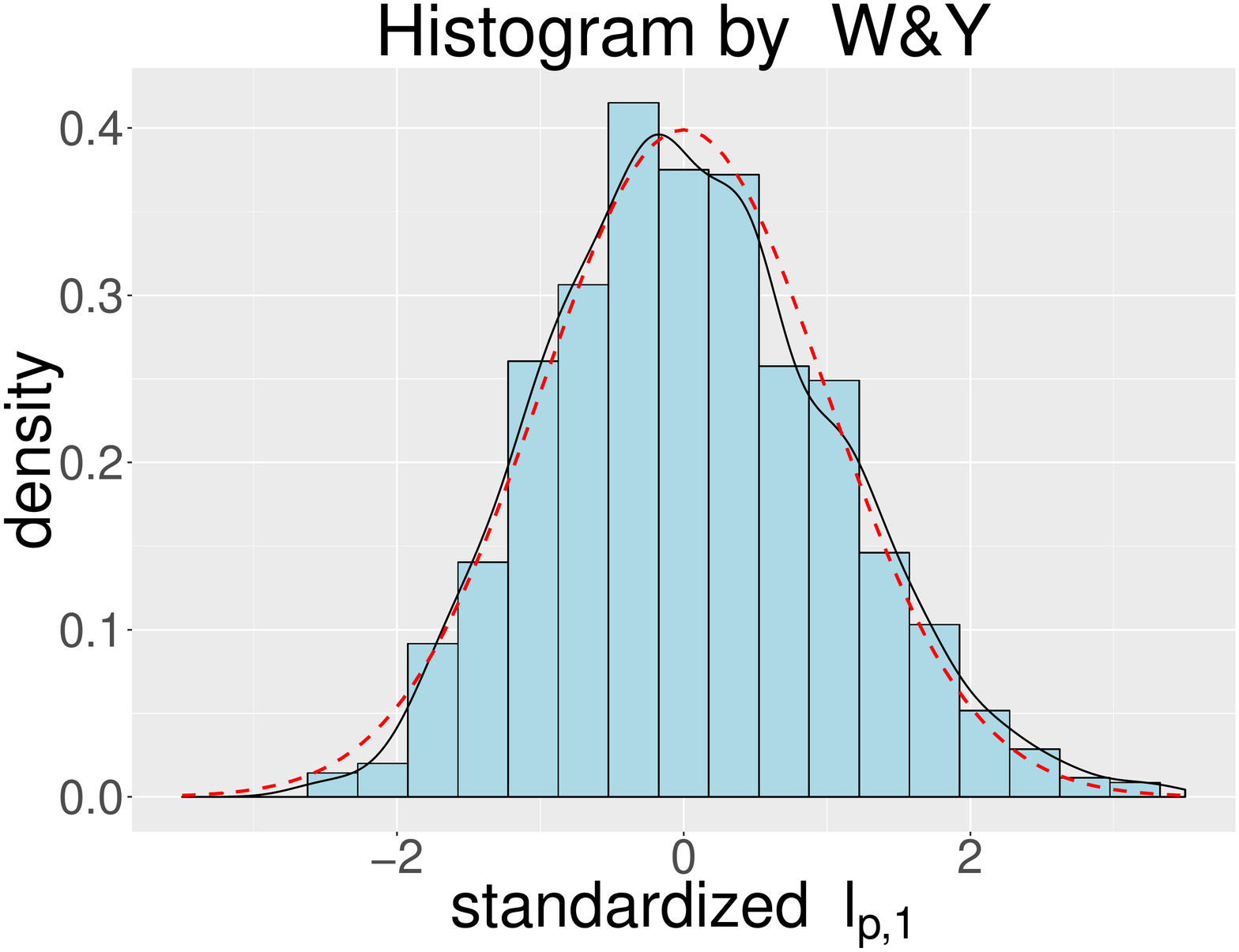}\quad
\includegraphics[width = .31\textwidth]{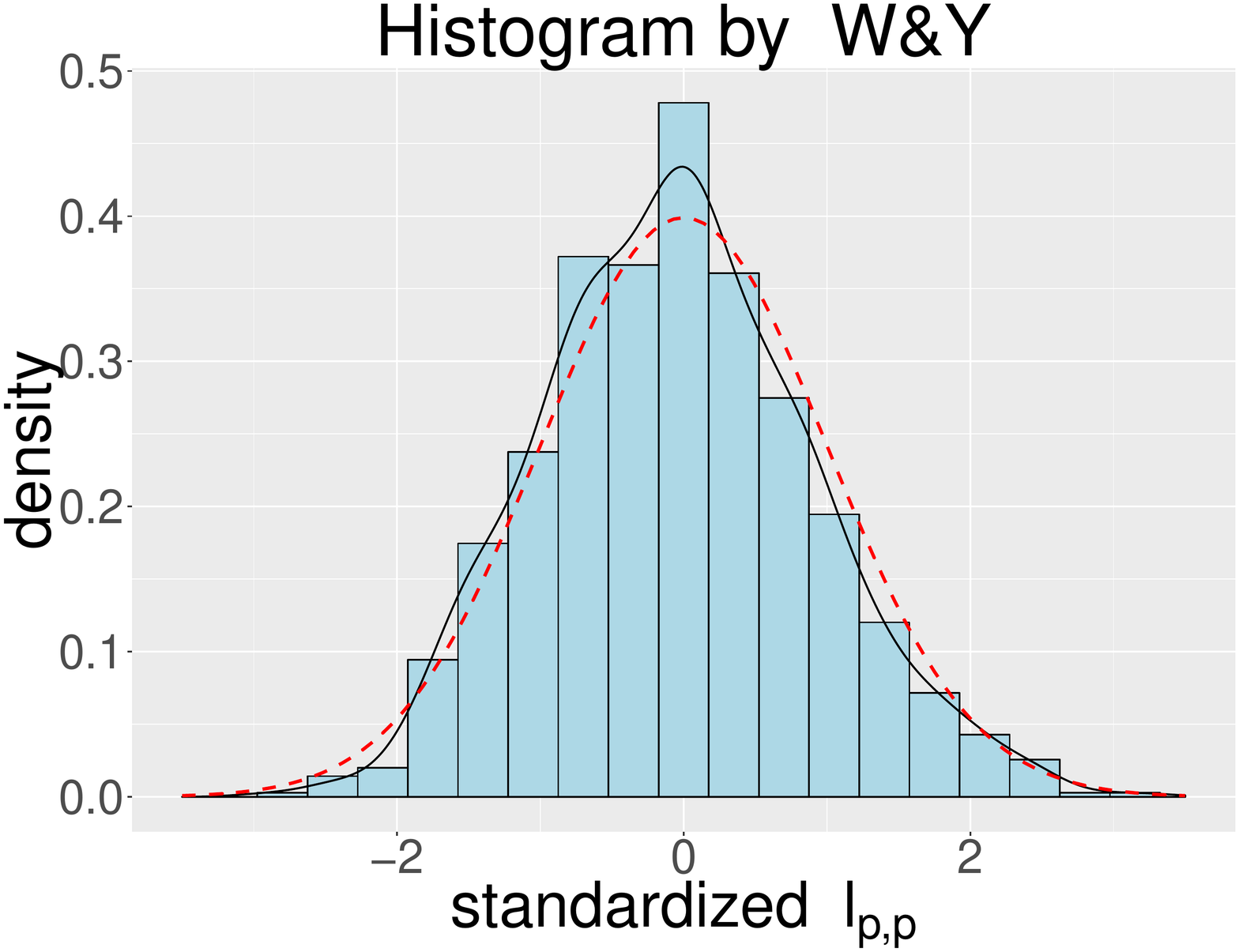}\quad
\includegraphics[width = .31\textwidth]{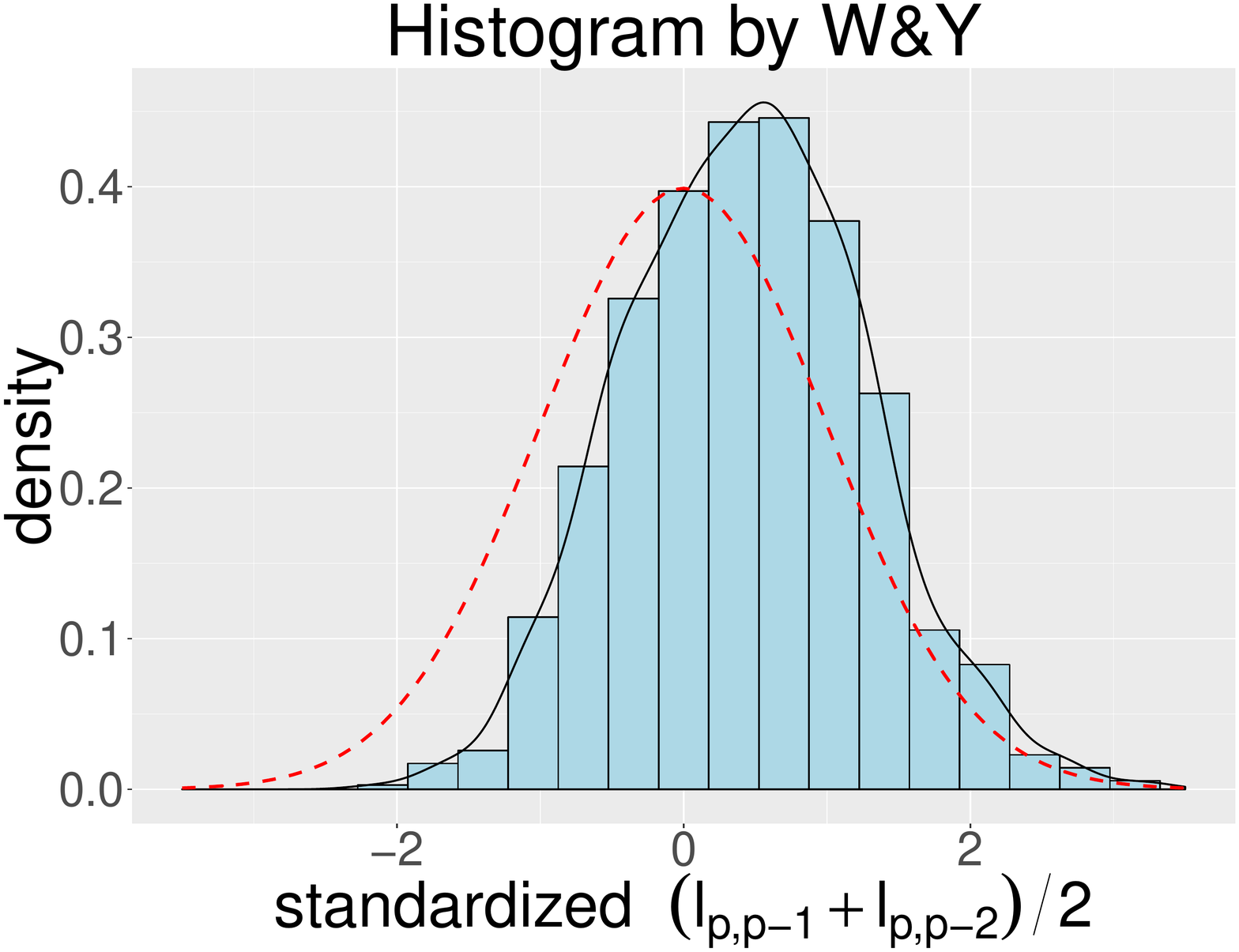}
\end{center}
\centerline{Case II}
\caption{Histograms of standardized estimated eigenvalues over 1000 simulations
when $x_{ij}\sim N(0,1)$ and $y_{ij}\sim N(0,1)$. The solid lines are the kernel density estimate,
and the dashed lines are the probability density function of $N(0,1)$}
\label{fig1}
\end{figure}

\begin{figure}[htbp]
\begin{center}
\includegraphics[width = .31\textwidth]{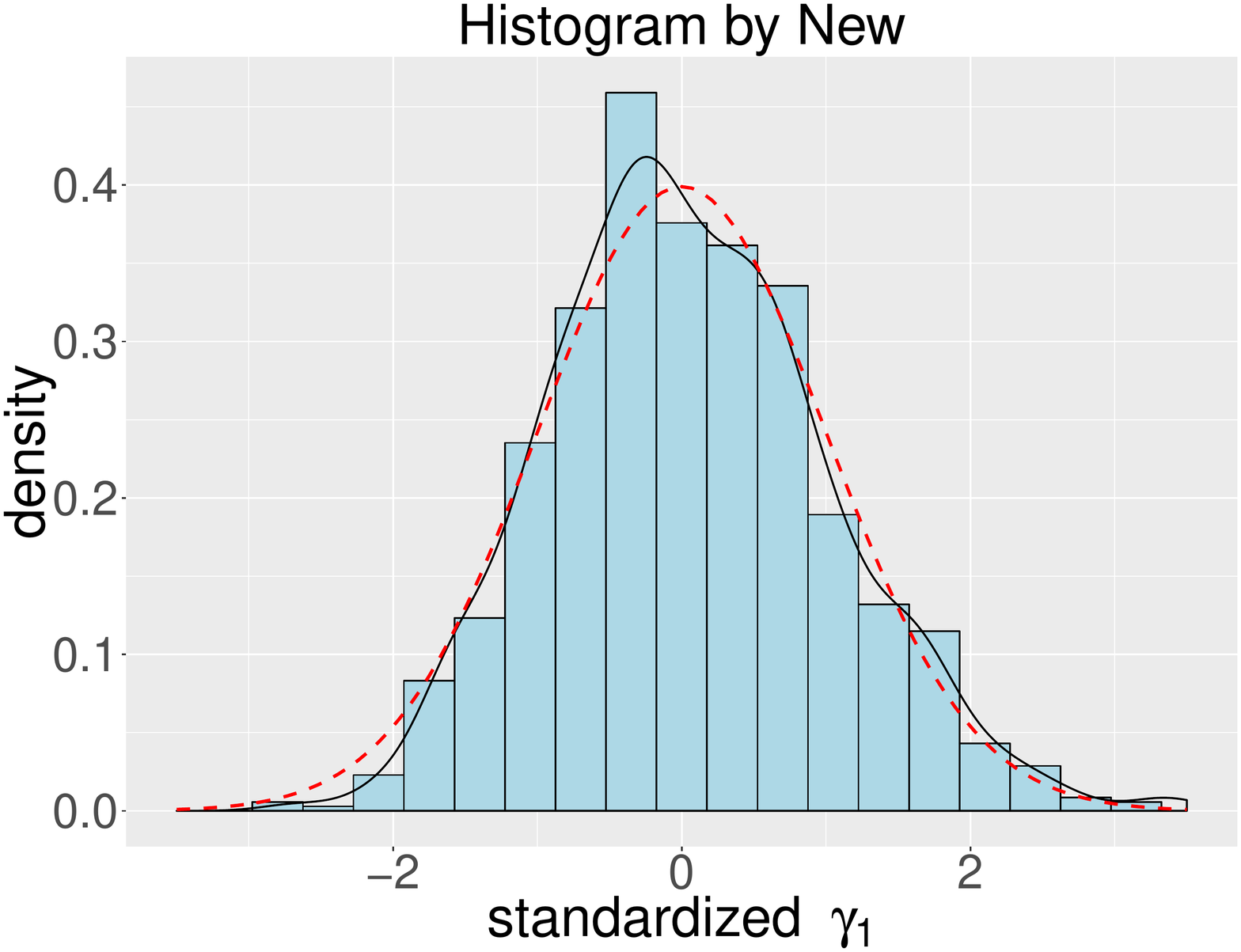}\quad
\includegraphics[width = .31\textwidth]{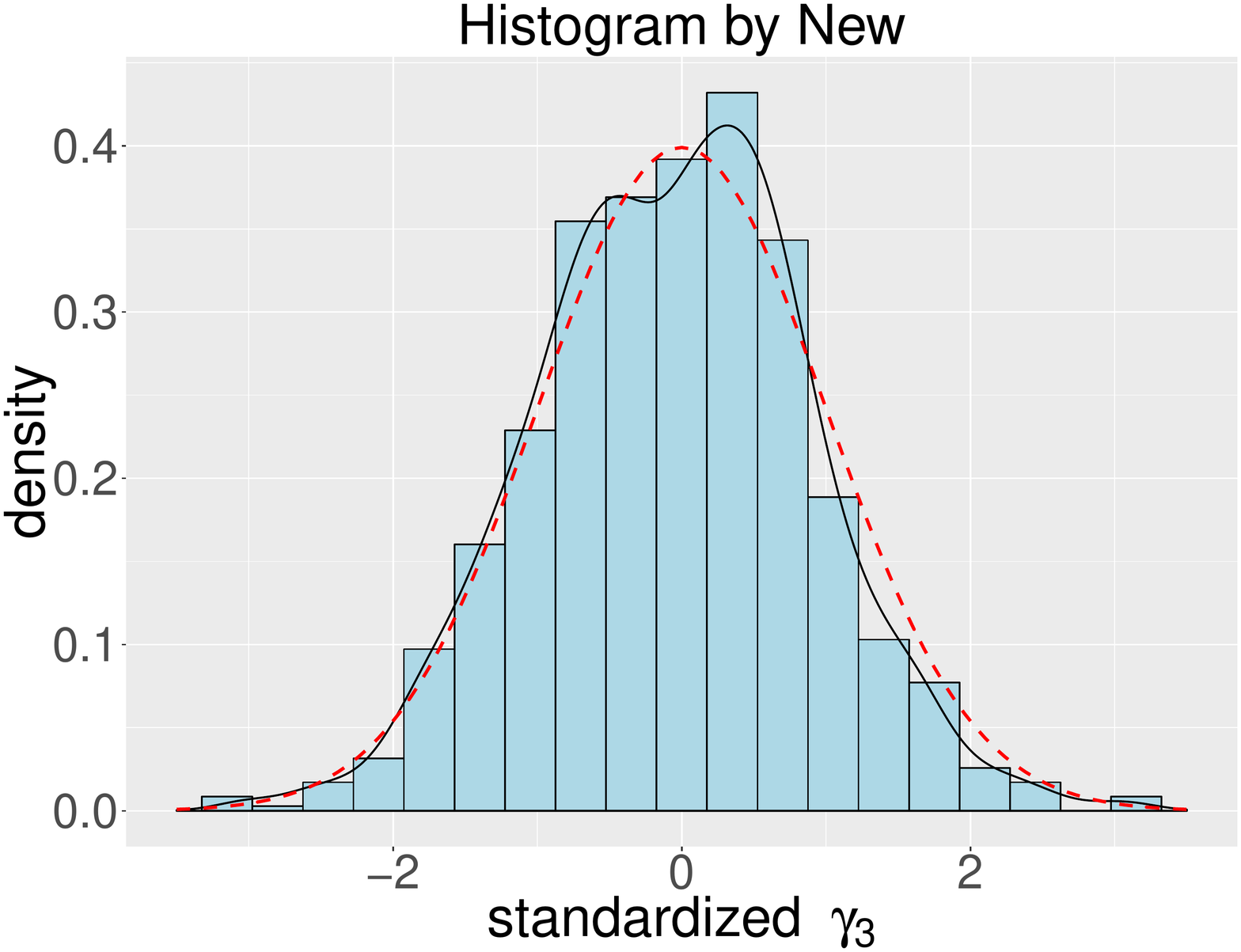}\quad
\includegraphics[width = .31\textwidth]{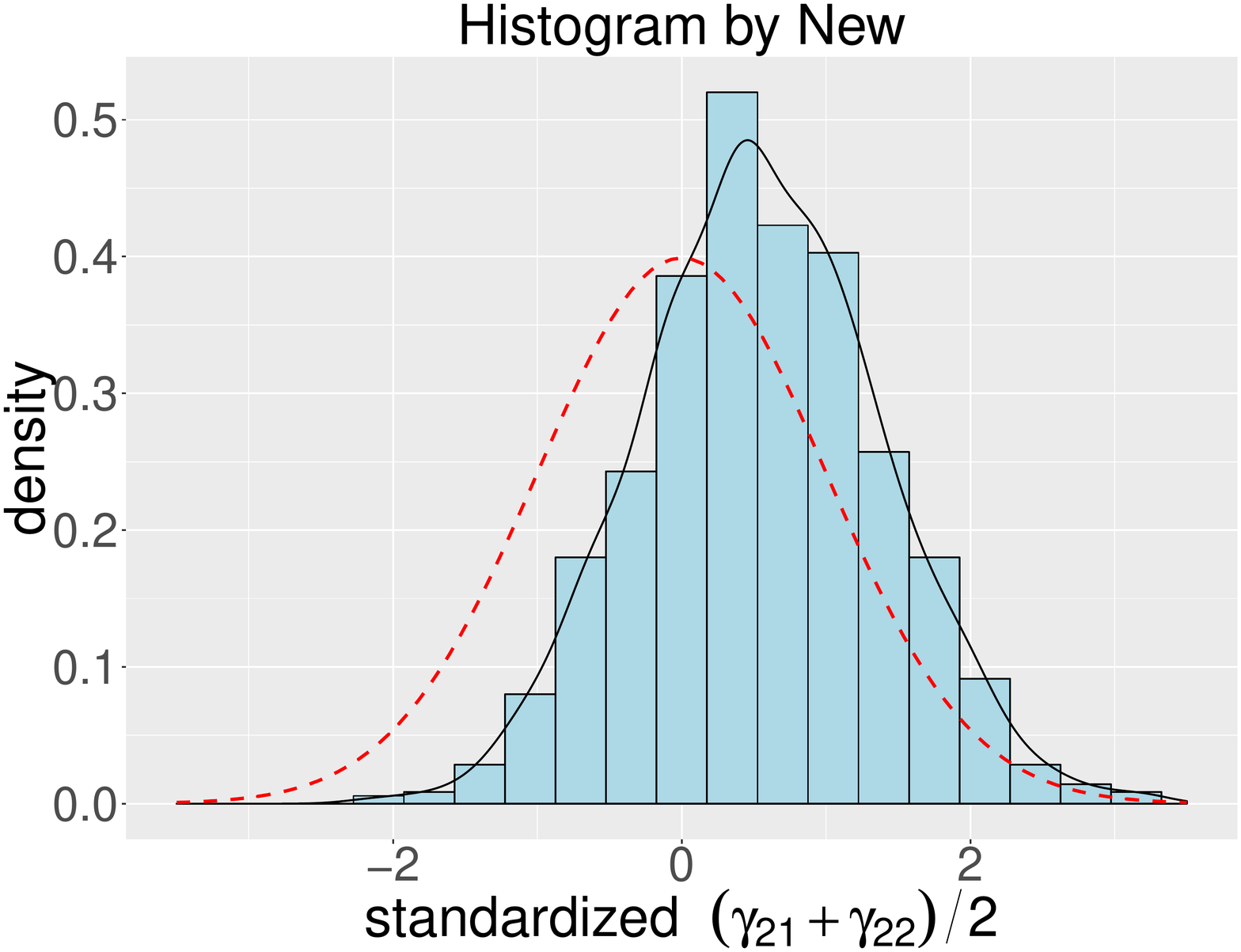}\\
\includegraphics[width = .31\textwidth]{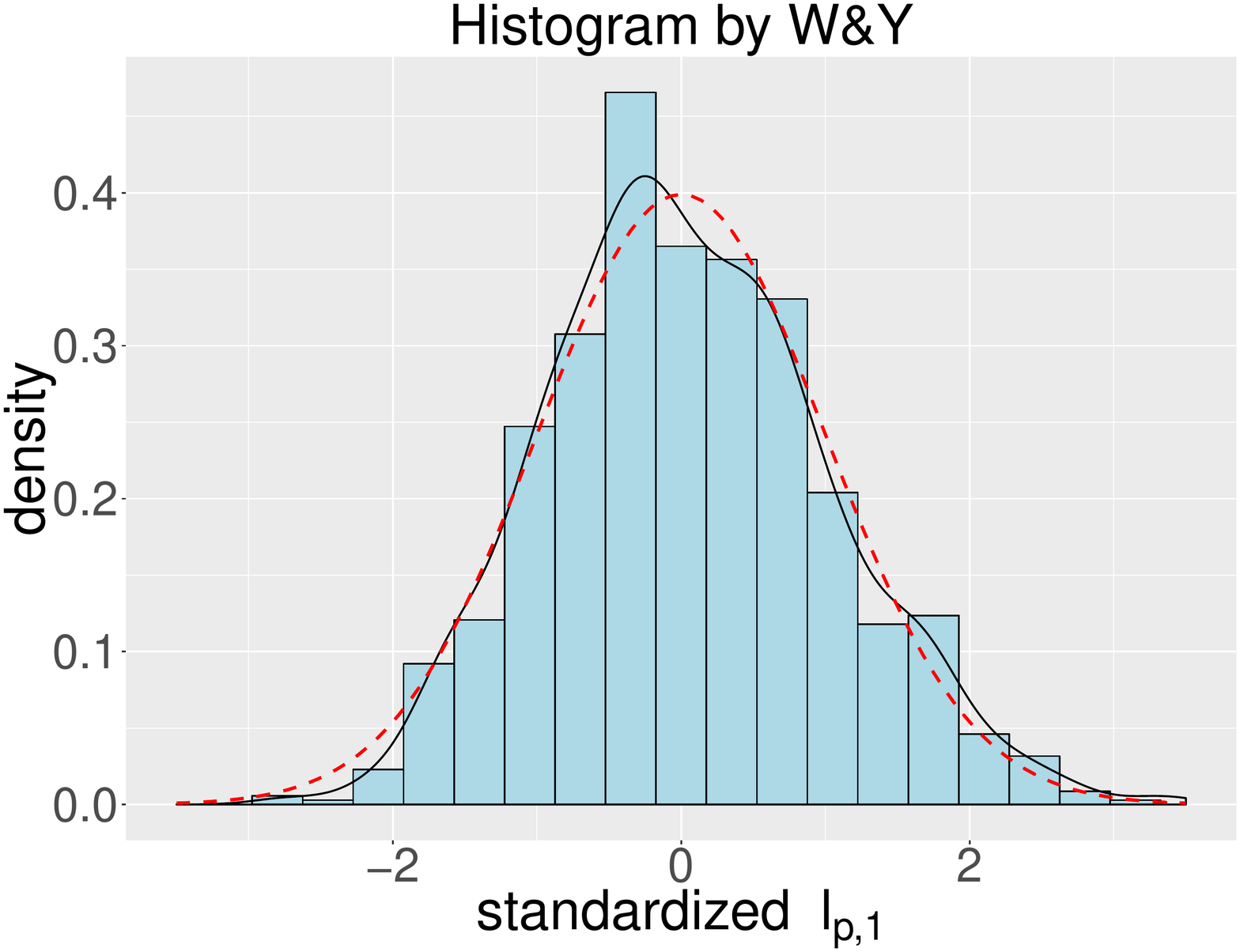}\quad
\includegraphics[width = .31\textwidth]{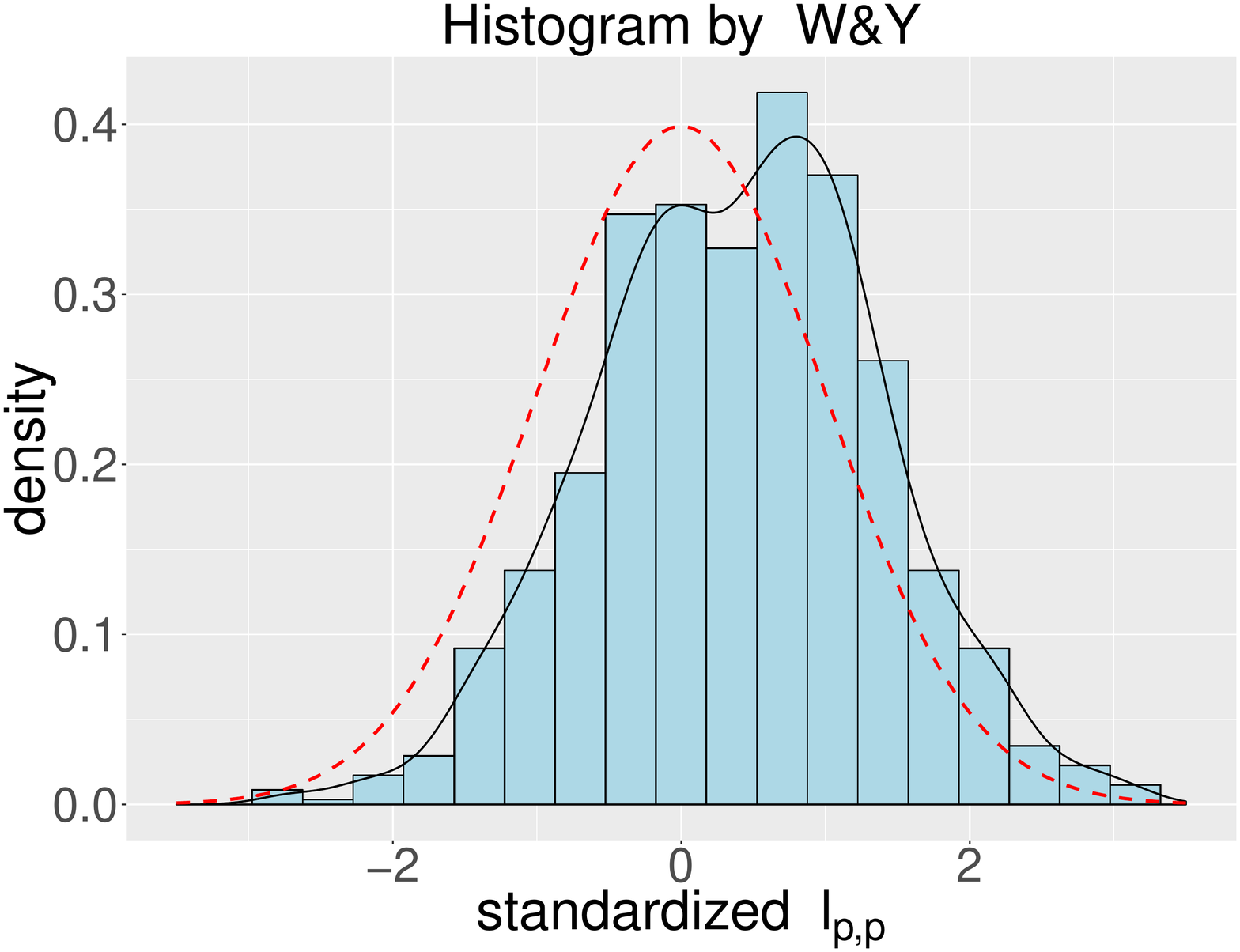}\quad
\includegraphics[width = .31\textwidth]{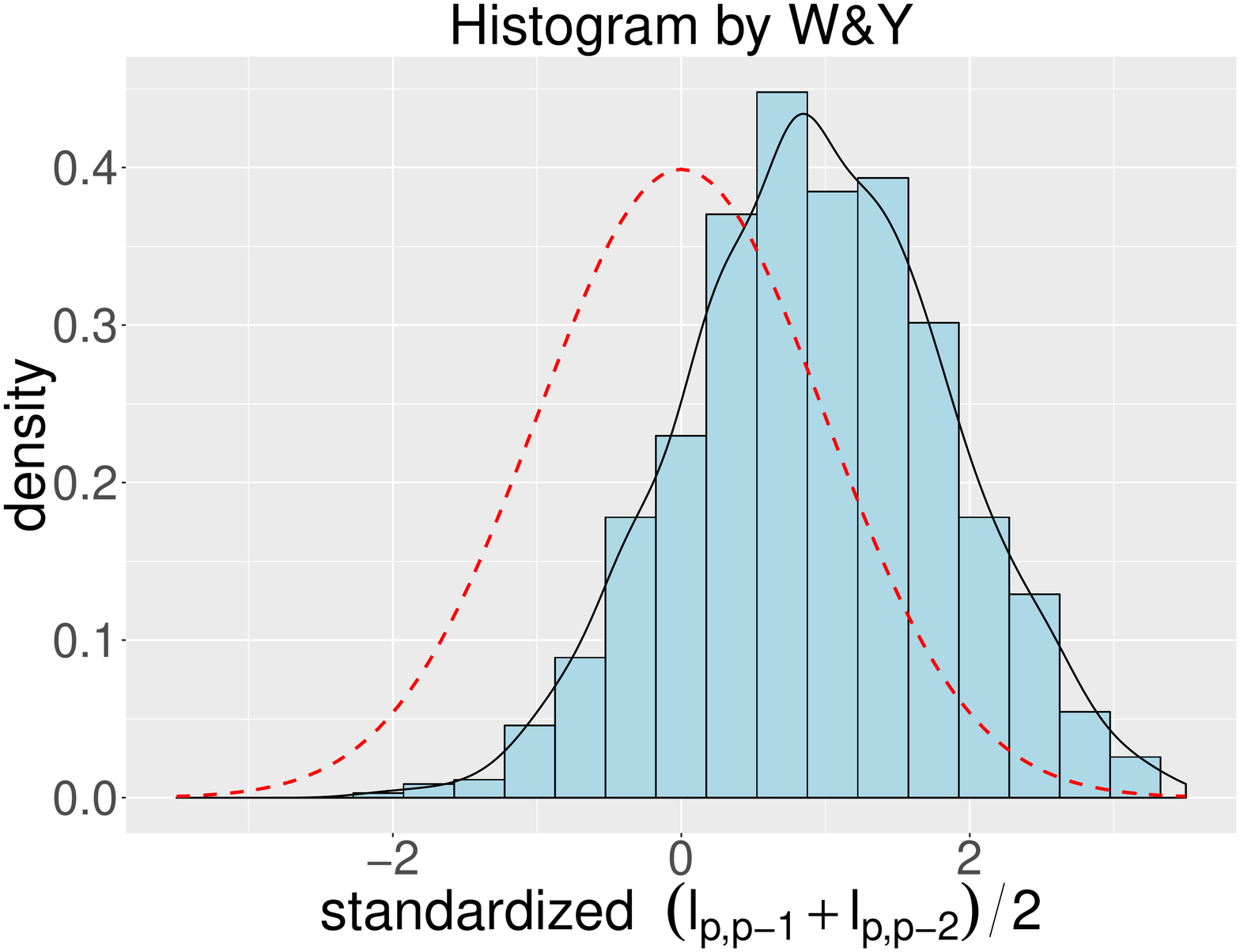}
\centerline{Case~I}
\end{center}
\begin{center}
\includegraphics[width = .31\textwidth]{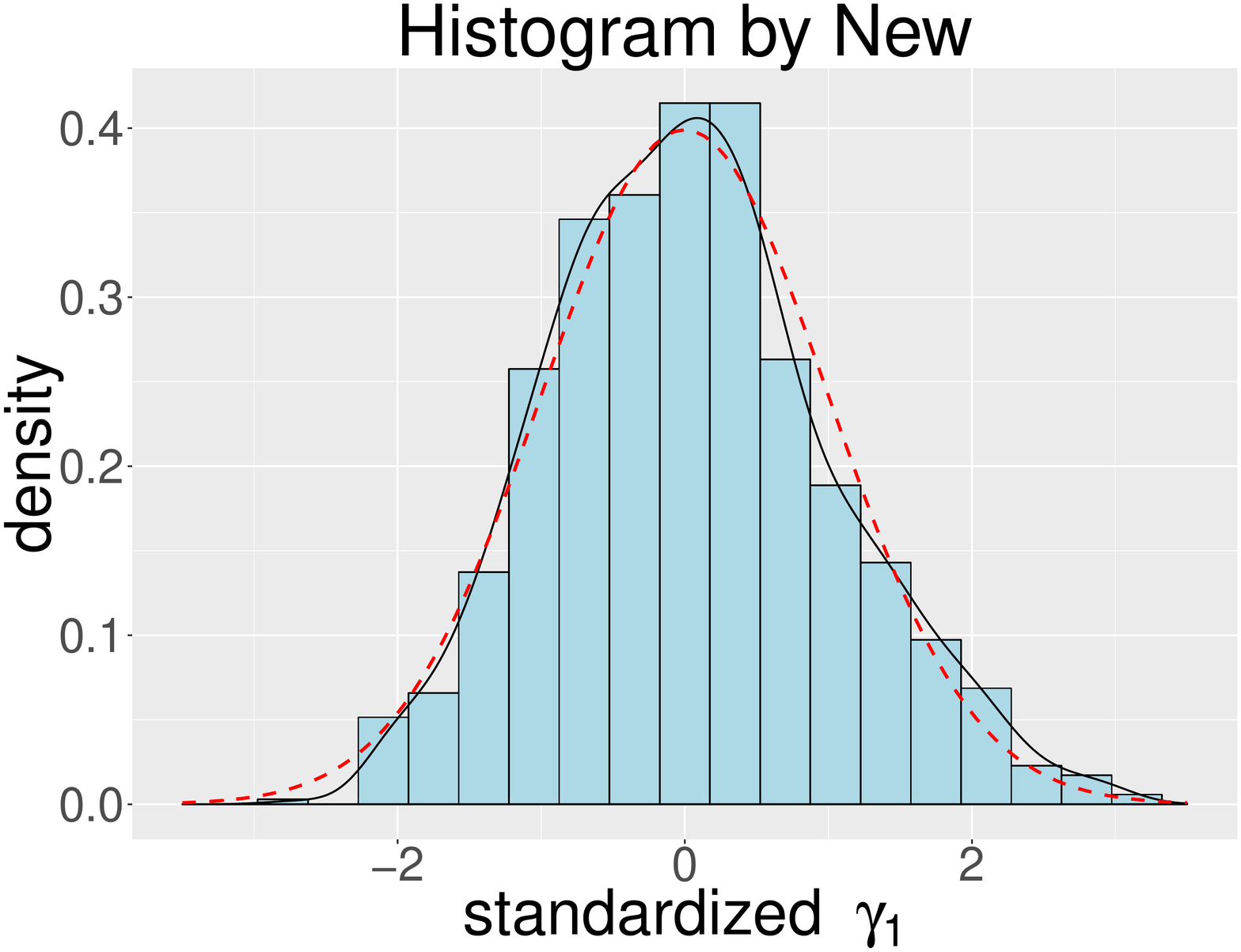}\quad
\includegraphics[width = .31\textwidth]{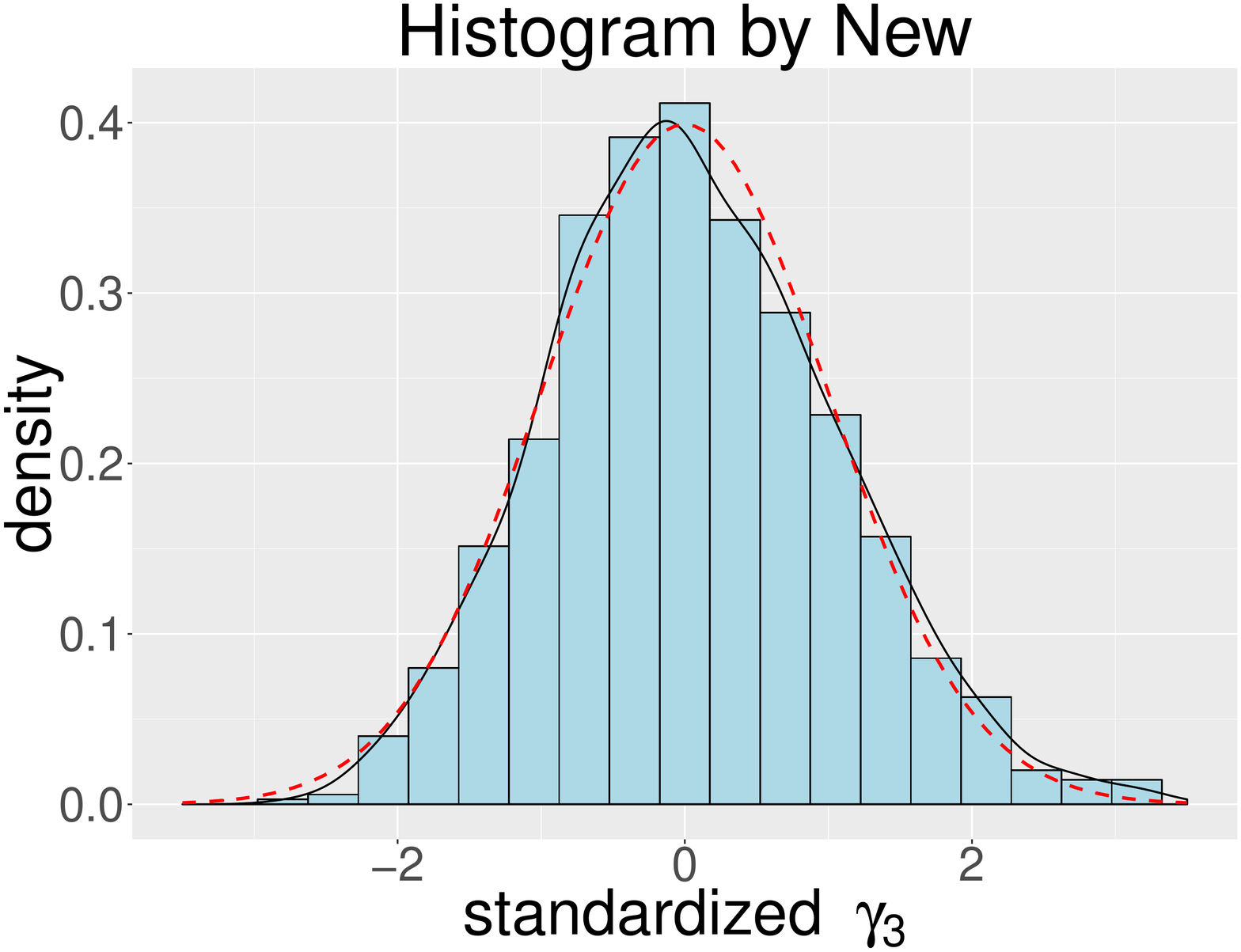}\quad
\includegraphics[width = .31\textwidth]{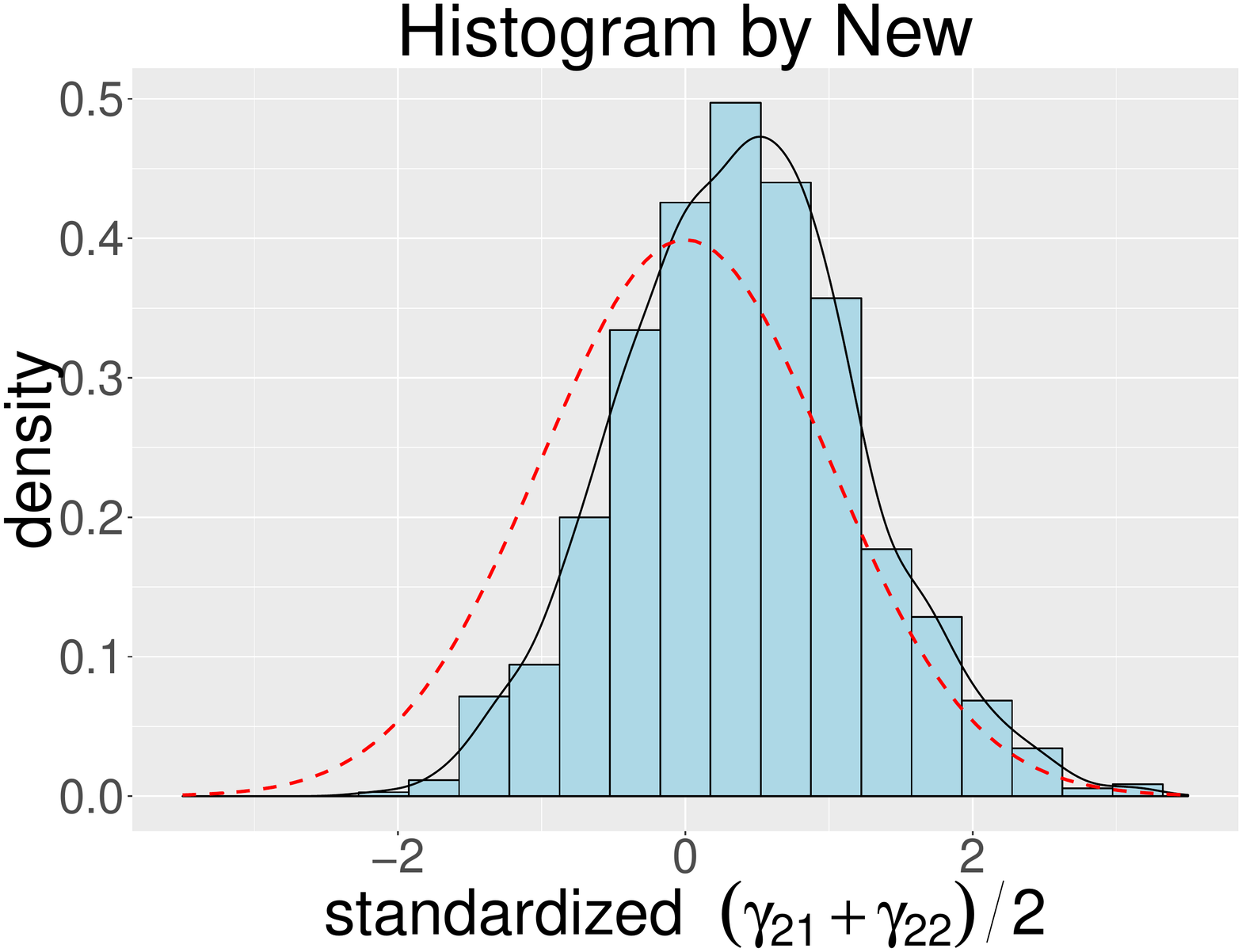}\\
\includegraphics[width = .31\textwidth]{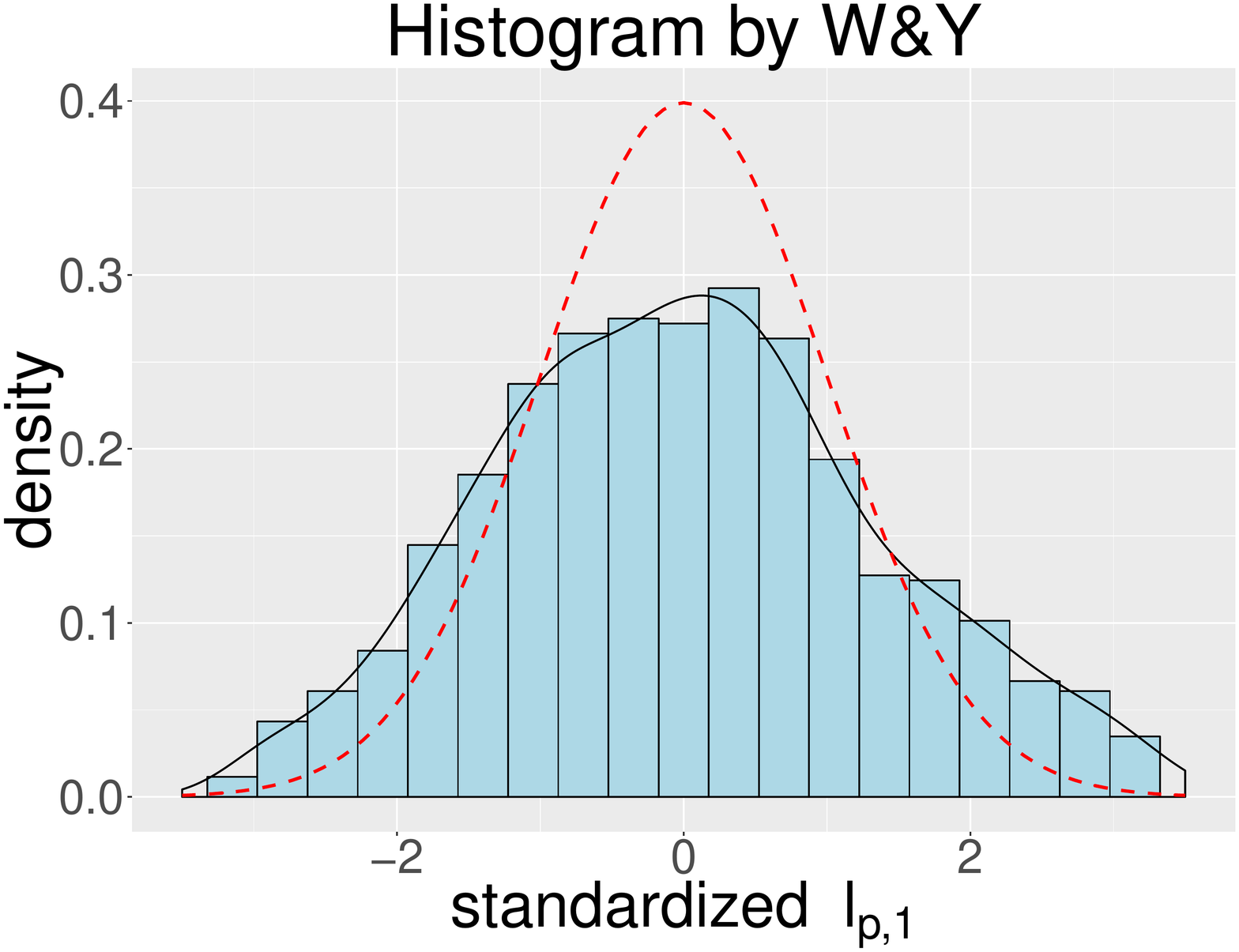}\quad
\includegraphics[width = .31\textwidth]{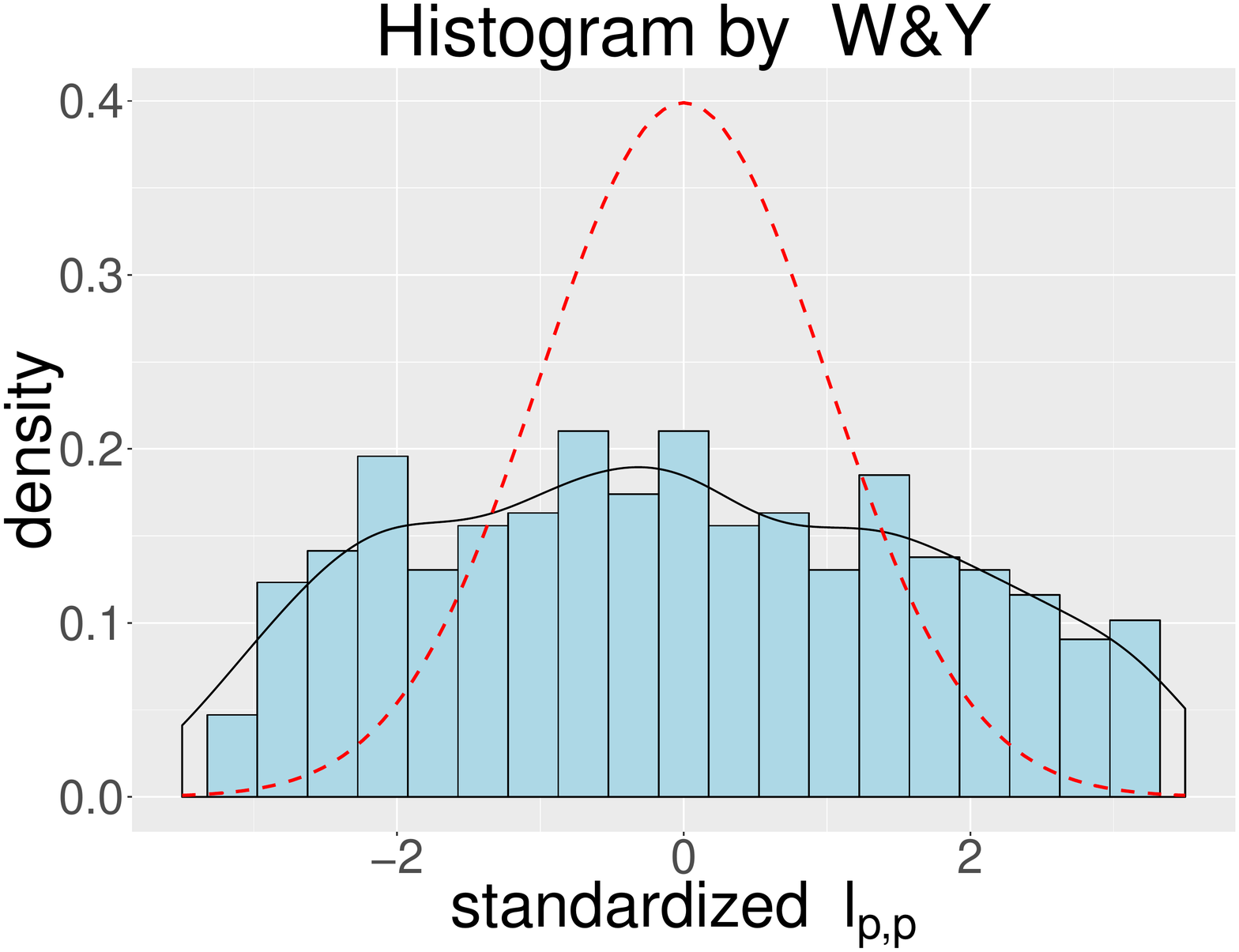}\quad
\includegraphics[width = .31\textwidth]{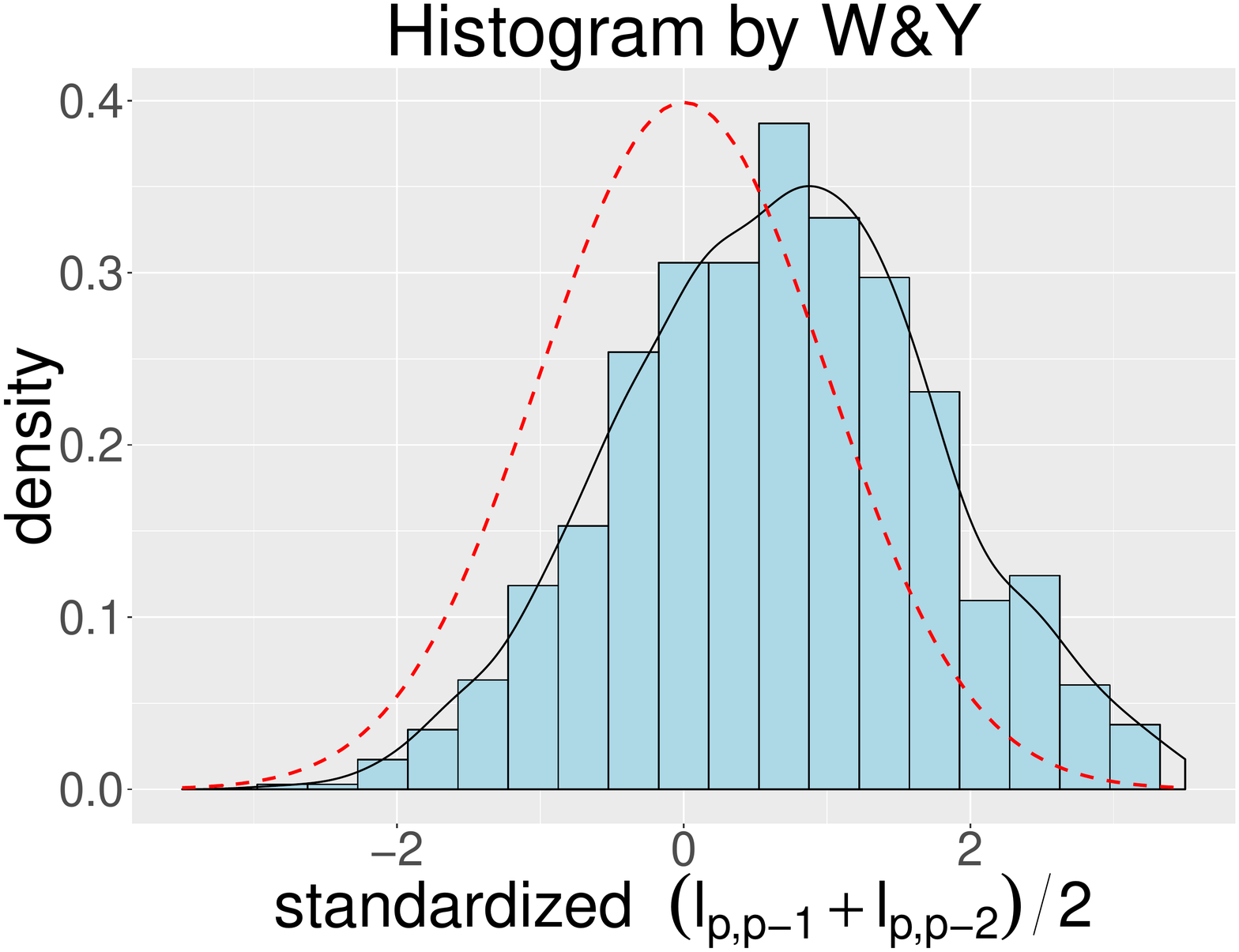}
\end{center}
\centerline{Case II}
\caption{Histograms of standardized estimated eigenvalues over 1000 simulation when
 $P(x_{ij}=\pm 1)=P(y_{ij}=\pm 1) =1/2$. Caption is the same as that in Figure~\ref{fig1}.}\label{fig2}
\end{figure}

  \begin{figure}[htbp]
\begin{center}
\includegraphics[width = .31\textwidth]{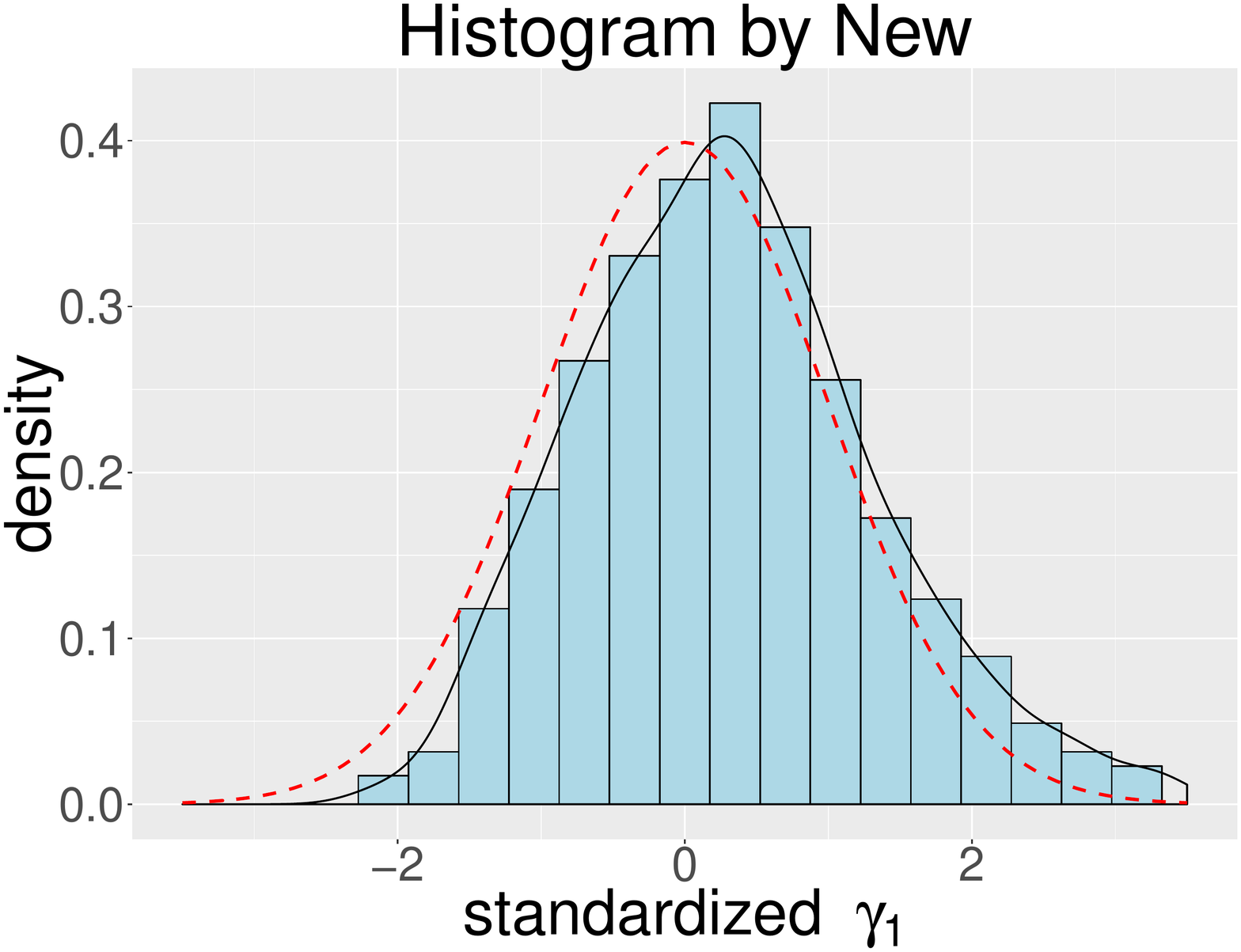}\quad
\includegraphics[width = .31\textwidth]{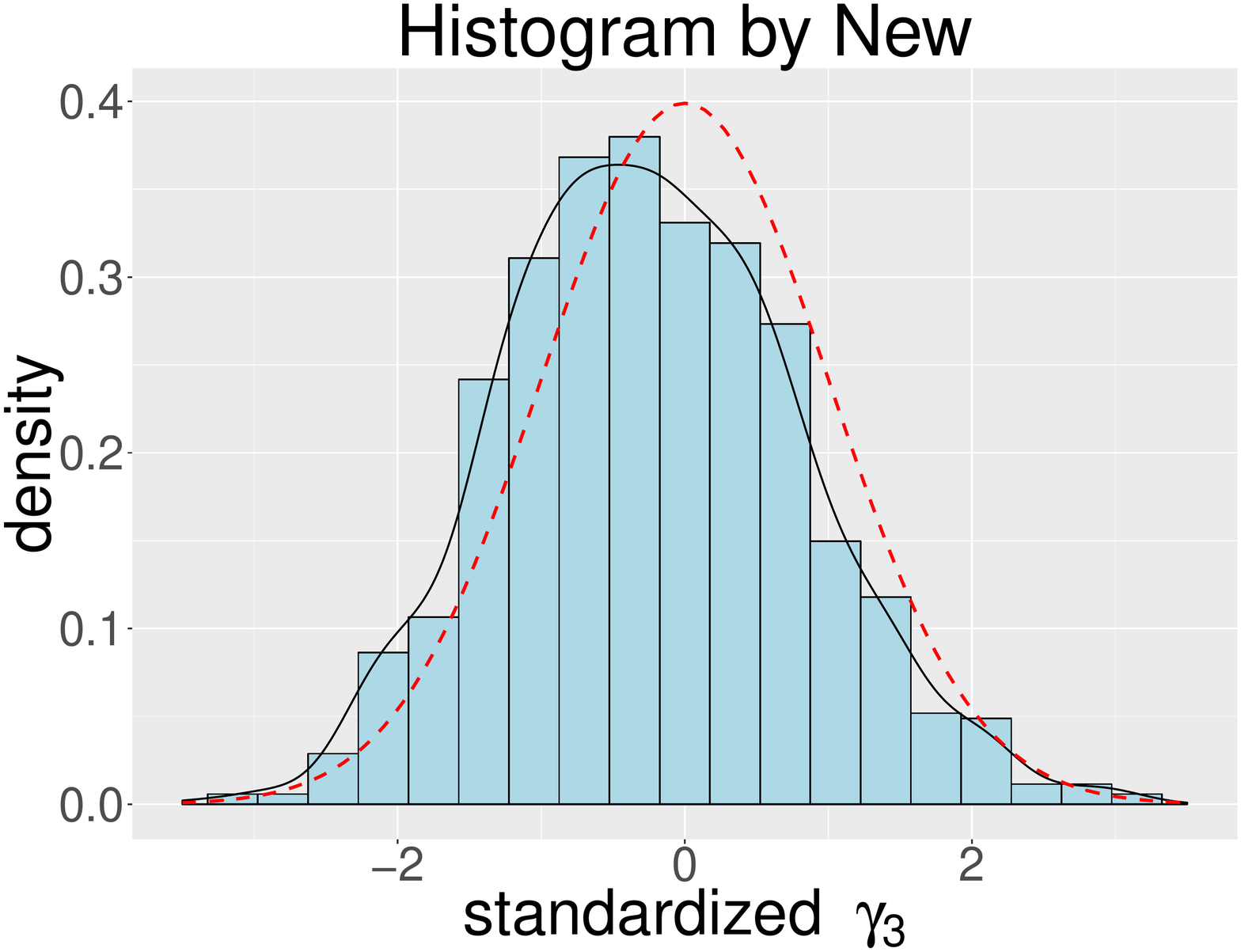}\quad
\includegraphics[width = .31\textwidth]{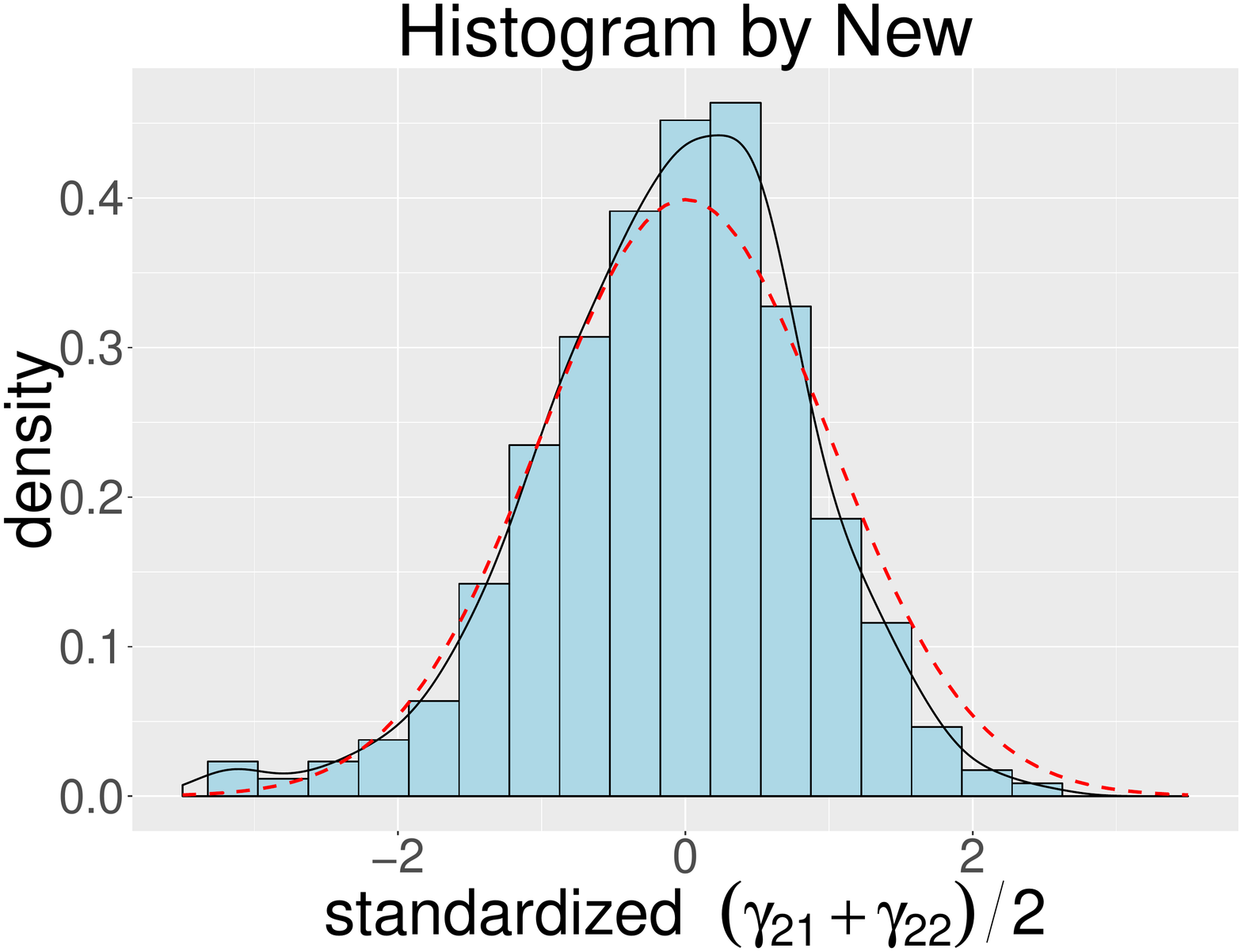}
\caption{Histograms of standardized estimated eigenvalues over 1000 simulations under Case~II when
 $x_{ij}$ and $y_{ij}$ follow the $2^{-1/2}t(4)$. Caption is the same as that in Figure~\ref{fig1}.}
 \label{fig3}
\end{center}
\end{figure}

For {\rm Case~I}, Remark~\ref{rmkclt} may be applied.   For the largest and least
single  population  spikes $\alpha_1=20$ and $\alpha_3=0.1$, the following CLTs are obtained,
\[\gamma_k= \sqrt{p-4}\Big(\frac{l_{p,j}(\mathbf F)}{\psi_{n,k}}-1\Big) \rightarrow N(0, \sigma_k^2),\]
where  $j=1$ for $k=1$ and $j=p$ for $k=3$,
$\psi_{n,1}=42.667, \sigma_1^2=2.383$ 
and $\psi_{n,3}= 0.0737,\sigma_3^2=1.343$ under the Gaussian Assumption; meanwhile, $\sigma_1^2=1.116$ and $\sigma_3^2=0.180$
for the distribution with binary outcome. To make it more accurate, we use $p-M$
instead of $p$ in the calculation.

For   the spikes $\alpha_2=0.2$ with multiplicity 2, we consider the sample eigenvalue $l_{1,p-1}$,
 $l_{1,p-2}$, and obtain that the two-dimensional random vector 
\[\gamma_2=\left(\gamma_{21}, \gamma_{22}\right)'=\left( \sqrt{p-4}\Big(\frac{l_{p,p-2}(\mathbf F)}{\psi_{n,2}}-1\Big),
\sqrt{p-4}\Big(\frac{l_{p,p-1}(\mathbf F)}{\psi_{n,2}}-1\Big)\right)' \]
converges to the eigenvalues of random matrix $-\left[\boldsymbol\Omega_{\psi_{2}}\right]_{22}/{\phi_k}$,
where  $\psi_{n,2}=0.133$, $\phi_k=1.439$ for the spike $\alpha_2=0.2$.  Furthermore, the
matrix
$\left[\boldsymbol\Omega_{\psi_{2}}\right]_{22}$
is a $2\times 2$ symmetric matrix with the independent Gaussian entries, of which the $(s,t)$ element has mean zero and the variance given by
$var(w_{st})=1.163$ if $s\neq t$ and $var(w_{st})=2.326$ if $s=t$
under the Gaussian Assumption. All the results are the same except $var(w_{st})=0.502$ if $s=t$ under the second population
(i.e. binary outcomes).
Under Case~II, it follows by Theorem~\ref{CLT} that the asymptotical means and covariances for all the three distributions considered in this simulation
are the same as the ones of  {\rm Case~I} under Gaussian assumption,
even if the population fourth moments are infinite, such as $t(4)$.


Let $F(u)$ be the cumulative distribution function (cdf) of a random variance $U$, and $F_n(u)$
be its empirical cdf based on a sample $u_1,\cdots,u_n$. Define Kolmogorov-Smirnov (KS)
statistic as follows:
\[
KS =\sum_u|F_n(u)-F(u)|
\]
In our simulation, we set $F(u)$ to be the cdf of $N(0,1)$, and $U$ be the standardized
estimated eigenvalues of the generalized spiked Fisher matrix. We evaluate $F_n(u)$ based
on 1000 simulations.

Table~\ref{tab1} depicts the  nine typical percentiles of the empirical distributions
of the standardized estimated eigenvalues, in which the $\psi_{n,k}$ and $\sigma_k^2$ were
calculated by two methods: one is derived in Theorem 3.1, and the other one is derived in
\cite{WangYao2017}, to compare the finite sample properties of our method and Wang and Yao's method
\cite{WangYao2017}.
The corresponding KS statistics of the two methods are also reported in Table~\ref{tab1}.
To see the overall pattern of the empirical distributions of the estimated spiked eigenvalues,
we present their histograms based on 1000 simulations along with the kernel density
estimate and asymptotic limiting distributions in Figures \ref{fig1}--\ref{fig3}.

Figure~\ref{fig1} depicts the histograms and kernel density estimate of estimated eigenvalues
when $x_{ij}$ and $y_{ij}$ follow $N(0,1)$. From Figure~\ref{fig1}, we can see that under Gaussian
population, both the empirical distributions of our method and Wang and Yao's method are
close to the asymptotical ones. This is further confirmed by the top panel of Table~\ref{tab1}.
The KS statistics of $\hat{\gamma}_1$ for Cases I and II and $\hat{\gamma}_3$ for Case II
are very close for the two methods. Our newly proposed method has
smaller KS statistics for $\hat{\gamma}_2^*$ for Cases I and II and $\hat{\gamma}_3$ for Case I
than Wang and Yao's method.

Figure~\ref{fig2} depicts the histograms and kernel density estimate of estimated eigenvalues
when $P(x_{ij}=\pm 1)=P(y_{ij}\pm1)=1/2$. Figure~\ref{fig2} clearly indicates that our method
works well for both Case I and II, while Wang and Yao's method works well for Case I,
but not for Case II. The middle panel of Table~\ref{tab1} also delivers the same message. Except
for $\hat{\gamma}_1$ in Case I, our method has much smaller KS statistics than Wang and Yao's method.
For the percentiles, it seems that Wang and Yao's method has more spreading-out percentile. This implies
it has larger variance.

Figure~\ref{fig3} depicts the histograms and kernel density estimate of estimated eigenvalues
when $x_{ij}$ and $y_{ij}$ follow $2^{-1/2}t(4)$. Figure~\ref{fig3} looks similar to the histograms
for Case II in Figure~\ref{fig1}. This implies that our method works well when the
fourth moment of population distribution is unbounded. This can be further confirmed
by comparing the bottom panel of Table~\ref{tab1} and the corresponding ones in the
top panel of Table~\ref{tab1}.

%

\begin{table}
\caption{Empirical sizes and powers} \label{Table51}
\begin{tabular}{@{}lccccccccc@{}}
\hline
 && \multicolumn{4}{c} {$q_1/q_0=0.2$} &    \multicolumn{4}{c}
 {$q_1/q_0=0.8$}\\
\cline{4-5} \cline{8-9}
 &&  \multicolumn{2}{c} {Size}&  \multicolumn{2}{c} {Power}& \multicolumn{2}{c} {Size}&  \multicolumn{2}{c} {Power}\\ [-1mm]
&$p$& {New}  & CLRT & {New} & CLRT  & {New}  & CLRT & {New} & CLRT\\
\hline
$\tilde c_1=5,\tilde c_2=0.8$;  & 50 &0.053 &0.058 &0.815 &0.632           &0.054               &0.069        &0.985             &0.794\\
    &100 &0.043  & 0.053 &1   &0.935                                       &0.051               &0.046        &1                    &0.899\\
      & 200 & 0.048  &N.A.  &1  &N.A.                                      &0.051               &N.A.          &1                    &N.A.\\
 $\tilde c_1=5, \tilde c_2=0.5$; & 50  &0.043   &0.057 &1   &0.998         &0.051              &0.049        &1                    &1\\
     &100  & 0.047 & 0.058 &1 &1                                           &0.040               &0.059        &1                    &1\\
       & 200  & 0.039 &N.A.  &1 &N.A.                                      &0.044               &N.A.          &1               &N.A.\\
  $\tilde c_1=5, \tilde c_2=0.2$; & 50  &0.034  &0.044  &1 &1              &0.030               &0.055        &1                     &1\\
      & 100 &  0.038 & 0.049  &1  &1                                       &0.043               &0.052        &1                     &1\\
      & 200 & 0.030  &N.A. &1  &N.A.                                       &0.042               &N.A.          &1                &N.A.\\
\hline
$\tilde c_1=2, \tilde c_2=0.8$;  & 50  & 0.042   &0.056 &{0.547}  &0.467   &0.051              &0.058        &0.838            &0.598\\
      & 100 & 0.041   & 0.043  &0.998  &0.680                              &0.052              &0.058        &1                     &0.752\\
      & 200 & 0.057   &N.A.  &1   &N.A.                                    &0.055              &N.A.          &1                &N.A.\\
 $\tilde c_1=2, \tilde c_2=0.5$;  & 50  & 0.035 & 0.053  &{0.996} &0.803   &0.042              &0.048        &1                     &0.988\\
      &100 & 0.052  & 0.056 &1  &0.995                                     &0.041              &0.052        &1                     &1\\
      & 200  &0.050  &N.A.  &1  &N.A.                                      &0.031              &N.A.          &1                &N.A.\\
 $\tilde c_1=2, \tilde c_2=0.2$;  & 50  &0.041  & 0.062  &1 &1             &0.041              &0.046        &1                    &1\\
      & 100  & 0.038 & 0.068  &1  &1                                       &0.043              &0.038        &1                     &1\\
      & 200& 0.047 &N.A.  &1  &N.A.                                        &0.038              &N.A.          &1                &N.A.\\
                                                                     \hline
$\tilde c_1=0.5, \tilde c_2=0.8$; & 50 & 0.046 & 0.047  &0.278  &0.386     &0.045               &0.049        &0.295             &0.371\\
     & 100  &0.055  & 0.049  &1  &0.658                                    &0.058               &0.054        &0.957              &0.574\\
     & 200 & 0.048  &N.A.  & 1  &N.A.                                      &0.043               &N.A.          &1                &N.A.\\
 $\tilde c_1=0.5, \tilde c_2=0.5$; & 50 & 0.047  & 0.048  &1  &0.699       &0.045                &0.043        &1
 &0.785\\
      & 100  & 0.031   &0.036  &1  &0.940                                  &0.044                &0.055        &1
      &0.927\\
     & 200  & 0.051   &N.A. & 1  &N.A.                                     &0.048               &N.A.          &1               &N.A.\\
  $\tilde c_1=0.5,\tilde c_2=0.2$; & 50  & 0.039 & 0.049  &1   &1          &0.045                 &0.044        &1                     &1\\
      & 100  & 0.037  & 0.047 &1   &1                                      &0.044                 &0.053        &1                     &1\\
      & 200  &  0.040 &N.A. &  1 &N.A.                                     &0.043              &N.A.            &1           &N.A.\\
\hline
\end{tabular}
\end{table}

\subsection{Numerical for Section \ref{sec5.1} } \label{sec5.2}

In this section, we conduct numerical study to compare the newly proposed test procedure for \eqref{HB0} in Section
\ref{sec5.1} with the  corrected likelihood ratio test (CLRT) proposed by
\cite{Baietal2013}. {\color{blue}The simulation results for signal detections in Section~{\ref{app2}}
are similar to those for \eqref{HB0}, we opt to omit them here to save space.}

For testing hypothesis \eqref{HB0},  we generate the elements of $\bB_2$ be
from $\mathcal{N}(1,1)$ for each simulation. Under the null hypothesis,  we set $\bB_{1}=\bo$, while
under  the alternative hypothesis, half of the entries in the first column of $\bB_1$ were generated from
$\mathcal{N}(0.5,1)$ and the rest are zeros.
Assume that the errors  $\bvare_i$ in (\ref{linear model}) follows
$\mathcal{N}_p(0,\bI_p)$. All elements of $\bz_i$ in
the model are independent and identically distributed and are sampled from $\mathcal{N}(1,0.5)$.

We consider two cases: $q_1/q_0=0.8$ and $q_1/q_0=0.2$. For each
case, set $p=50,100,200$, $\tilde c_1=0.5,2,5$ and $\tilde c_2=0.2, 0.5, 0.8$.
The limiting null distribution of Roy's test $\lambda_1(\bH\bG^{-1})$ follows
the Tracy-Widom law for $\tilde c_1/\tilde c_2 \lambda_1$
proposed by \cite{Hanetal2016}.
In order to avoid the calculation of the complex integral in the Tracy-Widom law,
we derive the explicit expressions of the Tracy Widom law for $\tilde c_1/\tilde c_2 \lambda_1$ by
Theorem~1.12  in \cite{Baoetal2018} based on the functional relationship of canonical
correlation matrix and Fisher matrix.
Then we report both empirical  sizes and powers with 1000 replications at  a significance level $\alpha=0.05$.
The simulation results are summarized in the Tables~\ref{Table51}.

The simulation illustrates that the limiting distribution of the Roy
test in the linear regression model provides good sizes and powers in our simulation settings.
Table~\ref{Table51} shows that
the Roy's test in this simulation setting seems to be more powerful than the CLRT.
As seen from Table~\ref{Table51}, the powers of the Roy test rapidly
increases to 1 as the sample size increase. For instance, for the
case of $q_1/q_0=0.2, p=50,\tilde c_1=2, \tilde c_2=0.8$ (i.e., $p=50,n=187,q_0=125,q_1=25$),
the power is  0.547 and  increases to 0.996 for the case of
$q_1/q_0=0.2, p=50,\tilde c_1=2, \tilde c_2=0.5$ (i.e., $p=50,n=225,q_0=125,q_1=25$).
In general, {\color{blue}both the Roy's test and the CLRT are  expected to have  good sizes}, but our approach has higher powers in most cases.
It is worth to noting that the CLRT cannot obtain the sizes and powers in the large-dimensional
setting, since the log-likelihood ratio involved in the test
statistic is approaching infinity for such cases. The limiting
distribution derived in Section 4 can still be used to calculate
the empirical powers even with the increasing dimension, while
the CLRT  fails when $p=200$.





\clearpage

%

%




\begin{funding}
The first author was supported by NSFC Grant 11971371 and  Natural Science Foundation of Shaanxi Province 2020JM-049.
\end{funding}
\clearpage
\newpage

%
%
%
%
%
%
%
\newpage

\end{document}